\documentclass[11pt]{article}
\usepackage[utf8]{inputenc}
\usepackage[english]{babel}
\usepackage{authblk}
\usepackage{graphicx}        
\usepackage{subfigure}                            
\usepackage{multicol}        
\usepackage[bottom]{footmisc}
\usepackage{placeins}
\usepackage{changes,todonotes}
\usepackage{pgf}
\usepackage{psfrag}
\usepackage{pgfplots}
\usepackage{hyperref}
\usepackage{appendix}
\usepackage{amsmath,amssymb}
\usepackage{bm}
\usepackage{mathtools}
\usepackage{color}
\usepackage{stmaryrd}
\usepackage{amsthm}
\usepackage{blindtext}
\usepackage{caption}
\usepackage{multirow}
\usepackage{float}

\usepackage[margin=2cm]{geometry}

\newcommand{\MB}[1]{\textcolor{black}{#1}}
\newcommand{\RRR}[1]{\textcolor{black}{#1}}

\newcommand{\ndof}{\mathsf{ndof}}
\newcommand{\dd}{\mathsf{d}}

\newcommand{\UU}{\mathsf{U}}
\newcommand{\VV}{\mathsf{V}}
\newcommand{\FF}{\mathsf{F}}
\newcommand{\GG}{\mathsf{G}}
\newcommand{\WW}{\mathsf{W}}
\newcommand{\ZZ}{\mathsf{Z}}
\newcommand{\EE}{\mathsf{E}}

\newcommand{\llbrace}{\lbrace\hspace{-0.15cm}\lbrace}
\newcommand{\rrbrace}{\rbrace\hspace{-0.15cm}\rbrace}
\newcommand{\norm}[1]{\| #1\|}
\newcommand{\trinorm}[1]{{\vert\kern-0.25ex\vert\kern-0.25ex\vert #1 \vert\kern-0.25ex\vert\kern-0.25ex\vert}}

\newtheorem{theorem}{Theorem}
\newtheorem{remark}{Remark}%
\newtheorem{definition}{Definition}

\newtheorem{corollary}{Corollary}
\newtheorem{lemma}{Lemma}[section]

\newtheorem{assumption}{Assumption}

\begin{document}

\title{Discontinuous Galerkin discretization of coupled poroelasticity-elasticity problems }

\author[$\star$]{Paola F. Antonietti}
\author[$\star$]{Michele Botti}
\author[$\star$]{Ilario Mazzieri}

\affil[$\star$]{MOX, Laboratory for Modeling and Scientific Computing, Dipartimento di Matematica, Politecnico di Milano, Piazza Leonardo da Vinci 32, I-20133 Milano, Italy}

\affil[ ]{\texttt {\{paola.antonietti,michele.botti,ilario.mazzieri\}@polimi.it}}

\maketitle

	
\begin{abstract}
This work is concerned with the analysis of a space-time finite element discontinuous Galerkin method on polytopal meshes (XT-PolydG) for the numerical discretization of wave propagation in coupled poroelastic-elastic media. The mathematical model consists of the low-frequency Biot's equations in the poroelastic medium and the elastodynamics equation for the elastic one. 
To realize the coupling, suitable transmission conditions on the interface between the two domains are (weakly) embedded in the formulation. 
The proposed PolydG discretization in space is then coupled with a dG time integration scheme, resulting in a full space-time dG discretization.  
We present the stability analysis for both the continuous and the semidiscrete formulations, and we derive error estimates for the semidiscrete formulation in a suitable energy norm. 
The method is applied to a wide set of numerical test cases to verify the theoretical bounds. Examples of physical interest are also presented to investigate the capability of the proposed method in relevant geophysical scenarios.
\end{abstract}

\maketitle
\section{Introduction}

The numerical simulation of wave propagation in heterogeneous media is an important aspect of a wide range of scientific problems including acoustic engineering \cite{TKWF2010}, vibro-acoustics \cite{krishnan}, aeronautical engineering \cite{castagnede1998ultrasonic}, biomedical engineering \cite{HAIRE1999291}, computational seismology \cite{carcione2014book}, and oil and gas exploration \cite{Zhang2019,Morency2011}.
In this framework, recent studies have been concerned with the propagation of seismic waves in coupled elastic and poroelastic media. For the region of reservoirs, Biot's model of poroelastic  wave equations is selected while for the background region \cite{biot1955theory}, the visco-elastic wave equation is employed, cf. \cite{Morency2008}.
Poroelastic-elastic problems model elastic waves impacting a porous material and consequently propagating through it. The coupling between the elastic and the poroelastic domain is a more general realization of the physically consistent transmission conditions discussed in \cite{David2021, Morency2008,Vashisth1991}. Indeed, in our case, partial filtration at the interface can be also taken into account.
Based on a second-order in-time displacement formulation, in this paper, we present a high-order space-time discontinuous Galerkin method on polytopal grids (XT-PolydG) for the discretization of a coupled poroelastic-elastic problem.
The theory developed completes the one presented in \cite{AntoniettiMazzieriNatipoltri2021} and \cite{bonaldi}, where coupled poroelastic-acoustic and elastic-acoustic problems were studied, respectively. 
We remark that the geometric flexibility, due to mild regularity requirements on the underlying polytopal mesh, together with the arbitrary-order accuracy featured by the proposed XT-PolydG method is fundamental within the applicative context under investigation as it guarantees: (i) \textit{flexibility} in the representation of the geometry; (ii) high level of \textit{accuracy} in both space and time dimensions; (iii) \textit{efficiency} for parallel computation. Finally, the coupling conditions between the elastic and the poroelastic domains are naturally incorporated (in a weak sense) in the proposed scheme.
 
In the literature, we can find many works concerning the numerical discretization of coupled poroelastic models. Here, we recall, e.g., the semi-analytical solution and the plane wave decomposition method \cite{lefeuve2012semi,peng2021benchmarking},
the Lagrange Multipliers method \cite{rockafellar1993lagrange,zunino,Flemisch2006}, the finite and boundary element methods \cite{BERMUDEZ200317,FKTW2010,ferro2006wave}, mixed and discontinuous finite elements \cite{bause2017space,phillips2008coupling,girault2011domain,Anaya2020}, the spectral and pseudo-spectral element method \cite{Morency2008,Sidler2010}, 
the finite difference method,  \cite{dai1995wave,wenzlau2009finite,masson2010finite,zhang20143d},
the ADER scheme \cite{ward2017discontinuous, delapuente2008,chiavassa_lombard_2013,zhan2020unified,Zhang2019,dupuy2011wave}, the virtual element method  \cite{burger2021virtual}, and the references therein.
Here, by taking inspiration from \cite{girault2011domain}, we analyze an Interior Penalty discontinuous Galerkin formulation of the coupled problem, where the interface conditions are naturally taken into account by the penalization terms. The geometrical flexibility and high-order accuracy of the proposed scheme are ensured by the use of polygonal meshes and by the dG discretization in the time dimension.   
We refer the reader to 
\cite{JOHNSON1993,Delfour81,Hughes88,FRENCH1993}
for early results on time dG schemes for wave type equations, to \cite{Vegt2006,AbPeHa06,DoFiWi16} for 
first-order hyperbolic problems,
to \cite{KrMo16,BaCaDiSh18} and to \cite{GoScWi17,MoPe18,PeScStWi20} for Trefftz and tent-pitching  techniques, 
 to \cite{banks2014high,Paper_Dg-Time,AntoniettiMiglioriniMazzieri2021} for the second-order wave equation, and to \cite{nochetto2018space} for a comprehensive review.
On the other hand, regarding polytopal dG methods, we refer the reader to \cite{AnBrMa2009, BaBoCoDiPiTe2012, AntoniettiGianiHouston_2013, CangianiDongGeorgoulisHouston_2016, CongreveHouston2019,cangiani2020hpversion, CangianiDongGeorgoulis_2017} for early results on elliptic and parabolic problems, to \cite{AntoniettiMazzieri2018, AntoniettiMazzieriMuhrNikolicWohlmuth_2020} for linear and non-linear hyperbolic problems, and to  \cite{bonaldi,AntoniettiBonaldiMazzieri_2019b,AntoniettiMazzieriNatipoltri2021,ABM_Vietnam} for coupled wave propagation problems. A dG approximation of the fully coupled thermo-poroelastic problem is presented in \cite{AntoniettiBonettiBotti_2023}. Wave propagation in  thermo-poroelastic media with dG methods is discussed in \cite{bonetti2023numerical}.\\

The remaining part of the paper is structured as follows. In Section~\ref{sec::physical} we present the coupled poroelastic-elastic differential model for wave propagation in heterogeneous media, we discuss the weak formulation and present the stability analysis in the continuous setting. In Section~\ref{sec::numerical}
we introduce the polytopal discontinuous Galerkin space discretization and analyze its stability.  The convergence analysis of the aforementioned discretization is discussed in Section~\ref{sec:semi_discrete_error}, while a dG time integration scheme and its algebraic formulation are presented in Section~\ref{sec:time_integration}. 
Section~\ref{sec:numerical_results} contains
verification test cases to validate the theoretical error bounds as well as numerical tests of physical interest. Finally, in
Section~\ref{sec:conclusions} we draw some conclusions and discuss some perspectives about future work.

\subsection*{Notation}
Let $\Omega\subset\mathbb{R}^d$, $d=2,3$, be an open, convex polygonal/polyhedral domain with Lipschitz boundary $\partial\Omega$.
In what follows, for $X\subseteq\Omega$, the notation $\bm{L}^2(X)$ is adopted in place of $[L^2(X)]^d$, with $d\in\{2,3\}$. The scalar product in $L^2(X)$ is denoted by $(\cdot,\cdot)_X$, with associated norm $\norm{\cdot}_X$.  
Similarly, $\bm{H}^\ell(X)$ is defined as $[H^\ell(X)]^d$, with $\ell\geq 0$, equipped with the norm $\norm{\cdot}_{\ell,X}$, assuming conventionally that $\bm{H}^0(X)\equiv\bm{L}^2(X)$. 
In addition, we will use $\bm{H}(\textrm{div},X)$ to denote the space of $\bm{L}^2(X)$ functions with square integrable divergence. For a given final time $T>0$, $k\in\mathbb{N}$, and a Hilbert space $H$, the usual notation $C^k([0,T];H)$ is adopted for the space of $H$-valued functions, $k$-times continuously differentiable in $[0,T]$. 
The notation $x\lesssim y$ stands for $x\leq C y$, with $C>0$, independent of the discretization parameters, but possibly dependent on the physical coefficients and the final time $T$.

\section{The physical model and governing equations}\label{sec::physical}
The computational domain $\Omega$ can be seen as the union of two disjoint, polygonal/polyhedral regions: $\Omega = \Omega_e\cup \Omega_p$, representing the elastic and the poroelastic domains, respectively. The two subdomains share part of their boundary, resulting in the Lipschitz-regular interface $\Gamma_I = \partial\Omega_e\cap\partial\Omega_p$, being $\partial\Omega_e $ and $\partial\Omega_p$ the boundaries of elastic and poroelastic domains, respectively. See Figure~\ref{fig:domain_example}. 
We set $\Gamma_{e}=\partial\Omega_e\setminus\Gamma_I$ and $\Gamma_{p}=\partial\Omega_p\setminus\Gamma_I$, so that $\Gamma_{i}\cap\Gamma_I=\emptyset$ for $i=\{e,p\}$.
For the sake of presentation, on $\Gamma_i$, $i=\{e,p\}$ we apply homogeneous Dirichlet conditions. The general case follows similarly. 
Additionally, we suppose that the Hausdorff measures of $\Gamma_{e}$, $\Gamma_{p}$, and $\Gamma_I$ are strictly positive, but the following theory also covers the cases: \textit{i)} $\Omega_p = \Gamma_I = \emptyset$, i.e. $\Omega\equiv \Omega_e$ and \textit{ii)} $\Omega_e = \Gamma_I = \emptyset$, i.e. $\Omega \equiv \Omega_p$; as well as \textit{iii)} $\partial\Omega_p=\Gamma_I$, i.e. $\Gamma_p=\emptyset$ and \textit{iv)} $\partial\Omega_e=\Gamma_I$, i.e. $\Gamma_e=\emptyset$. 
\begin{figure}
    \centering
    \includegraphics[width=0.45\textwidth]{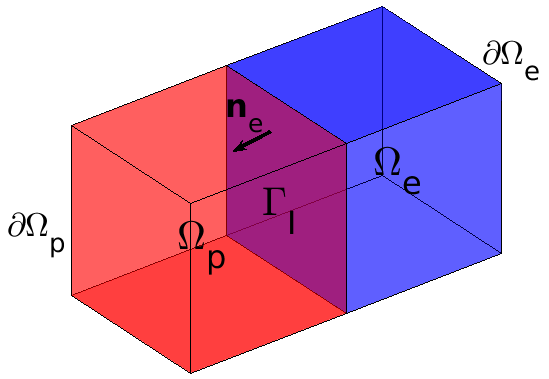}
    \includegraphics[width=0.4\textwidth]{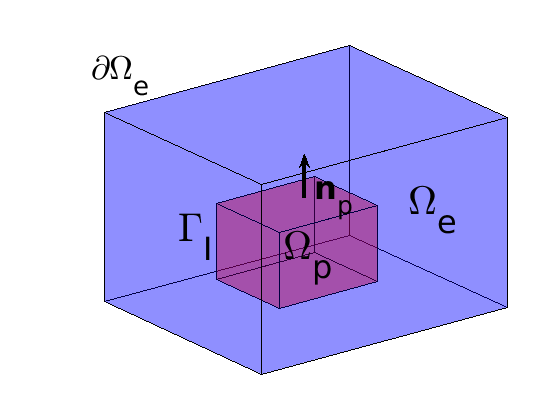}
    \caption{Simplified representations of the domain $\Omega = \Omega_p \cup \Omega_e$. Here,  $\Gamma_I = \partial \Omega_e \cap \partial \Omega_p$ represents the interface between the elastic and the poroelastic domains, where $\partial\Omega_e$ and $\partial\Omega_p$ are the boundaries of elastic and poroelastic sub-domains, respectively.}
    \label{fig:domain_example}
\end{figure}
The outer unit normal vectors to $\partial\Omega_e$  and $\partial\Omega_p$ are denoted by $\bm{n}_e$, and $\bm{n}_p$ respectively, so that $\bm{n}_p=-\bm{n}_e$ on $\Gamma_{I}$.

\subsection*{Elastic domain} 
In the solid elastic domain $\Omega_e$, we consider the linear visco-elastodynamics model
\begin{equation}
\rho_e\ddot{\bm{u}}_e + 2\rho_e \zeta \dot{\bm{u}}_e  + \rho_e\zeta^2 \bm{u}_e -  \nabla\cdot\bm{\sigma}_e(\bm{u}_e)=\bm{f}_e,  \quad \text{in }\Omega_e\times(0,T],
\label{eq::elasticity}
\end{equation}
where $\bm u_e$ represents the solid displacement, $\rho_e >0$ is the medium density, $\zeta > 0$ is an attenuation parameter \cite{ABM_Vietnam}, and  $\bm{f}_e$ is a given external load. The stress tensor $\bm{\sigma}_e(\bm{u})$ is defined as $\bm{\sigma}_e(\bm{u}) = \mathbb{C}:\bm{\epsilon}(\bm{u})$,  where
 $\bm{\epsilon}(\bm{u})=\frac{1}{2}(\nabla\bm{u}+\nabla\bm{u}^T)$ is the strain tensor, and $\mathbb{C}$ is the fourth-order, symmetric and uniformly elliptic elasticity tensor defined by
 \begin{equation*}
\mathbb{C}:\bm{\tau} = 2\mu \bm{\tau} + \lambda \rm{tr}(\bm\tau) \qquad\text{for all } \bm{\tau}\in\mathbb{R}^{d\times d},
\end{equation*}
with $\rm{tr}(\bm{\tau}) = \sum_{i=1}^d \bm{\tau}_{ii}$.
Here, $\lambda\ge0$ and $\mu\ge\mu_0>0$ are the Lam\'e coefficients.

\subsection*{Poroelastic domain} In the poroelastic domain $\Omega_p$, following \cite{AntoniettiMazzieriNatipoltri2021}, we consider the low-frequency Biot's equations:
\begin{equation}
\begin{cases}
\rho_p\ddot{\bm{u}}_p + 
\rho_f\ddot{\bm{u}}_f
+ 2\rho_p \zeta \dot{\bm{u}}_p  + \rho_p\zeta^2 \bm{u}_p
-\nabla\cdot\bm{\sigma}_p(\bm{u}_p,\bm{u}_f)=\bm{f}_p,  &\text{in }\Omega_p\times(0,T],\\[5pt]
\rho_f\ddot{\bm{u}}_p + 
\rho_w\ddot{\bm{u}}_f + 
\frac{\eta}{k}\dot{\bm{u}}_f
+\nabla p(\bm{u}_p,\bm{u}_f)=\bm{g}_p,  &\text{in }\Omega_p\times(0,T].
\end{cases}
\label{eq::poroel}
\end{equation}
Here, $\bm{u}_p$ and $\bm{u}_f$ represent the solid and filtration displacements, respectively. In  \eqref{eq::poroel} the average density $\rho_p$ is given by $\rho_p=\phi\rho_f+(1-\phi)\rho_s$, where $\rho_s>0$ is the solid density, $\rho_f>0$ is the saturating fluid density,  $\rho_w$ is defined as $\rho_w=\frac{a}{\phi}\rho_f$, being $\phi$ the \textit{porosity} satisfying $0<\phi_0\leq\phi\leq\phi_1<1$, and being $a>1$ the \textit{tortuosity} measuring the deviation of the fluid paths from straight streamlines, cf. \cite{souzanchi2013tortuosity}. The dynamic \textit{viscosity} of the fluid is represented by $\eta>0$  while the absolute \textit{permeability} by $k>0$. In  \eqref{eq::poroel}, $\textbf{f}_p$ and $\textbf{g}_p$ are given (regular enough) loading and source terms, respectively.
%
%
In $\Omega_p$,  we assume the following constitutive laws, cf. \cite{ezziani}, for the stress $\bm{\sigma}_p$ and pressure $p$:
\begin{align}
& \bm{\sigma}_p(\bm{u}_p,\bm{u}_f)= \bm{\sigma}_e(\bm{u}_p) 
-\beta\, p(\bm{u}_p,\bm{u}_f) \bm{I}, 
&& p(\bm{u}_p,\bm{u}_f) = -m(\beta \nabla\cdot\bm{u}_p+\nabla\cdot \bm{u}_f),
\label{eq::const_sigma_press}
\end{align}
where the Biot--Willis's coefficient $\beta$ and the  Biot's modulus  $m$  are such that $\phi<\beta\le1$ and $m\ge m_0>0$. It can be shown that the dilatation coefficient of the saturated matrix corresponds to $\lambda_f=\lambda+\beta^2m$. 
A summary of all the model coefficients together with their physical meaning and unit of measure is given in Table \ref{tab::table_poroelastic} below.  
To be consistent, in \eqref{eq::poroel} we consider the same viscous model as in \eqref{eq::elasticity}. We remark that other viscous models can be considered in both equations, see e.g. \cite{Morency2008}, but are beyond the scope of this work.

\subsection*{Poroelastic-elastic coupling} On $\Gamma_{I}$ the coupling between the elastic and the poroelastic domains is achieved by imposing the following transmission conditions, expressing continuity of the normal stresses, the continuity of displacements, and the absence of fluid flow into the elastic domain:
\begin{align}
\bm{\sigma}_e(\bm u_e)\bm{n}_p & = \bm{\sigma}_e(\bm{u}_p)\bm{n}_p
- \delta \beta\, p(\bm{u}_p,\bm{u}_f) \bm{I}\bm{n}_p,  &\text{in }\Gamma_I\times(0,T],& \label{eq::contstress_elporo}\\
\bm{u}_e &= \bm{u}_p,	&\text{in }\Gamma_I\times(0,T],& \label{eq::cont_disp} \\
 ((1-\delta)\beta \bm u_p  + \bm{u}_f) \cdot \bm{n}_p & = 0, &\text{in }\Gamma_I\times(0,T],& \label{eq:noflowrate}
\end{align}
for $\delta \in [0,1]$ representing the fluid entry resistance of the interface. In the case $\delta=1$, condition \eqref{eq:noflowrate} reduces to $\bm{u}_f \cdot \bm{n}_p=0$, namely there is no filtration through the interface; whereas if $1-\delta=\phi\beta^{-1}<1$, then \eqref{eq:noflowrate} imposes that the normal component of the fluid displacement $\bm{u}_p+\phi^{-1}\bm{u}_f$ is zero, meaning that the fluid is partially flowing into the interface.  
This assumption on the coupling conditions is somewhat similar to the one that is considered \cite{AntoniettiMazzieriNatipoltri2021} for the poroelastic--acoustic case. For the latter, the continuity of the pressure at the interface was a function of the parameter $\tau \in [0,1]$ modeling the closing/opening of the pores at the interface.

\subsection*{Coupled problem}
Supplementing the previous constitutive equations with homogeneous Dirichlet boundary conditions, the coupled \textit{poroelastic-elastic problem} reads: 
for any $t \in (0,T]$, find the vector-fields $\UU = (\bm{u}_e,\bm u_p, \bm{u}_f): \Omega_e \times \Omega_p \times \Omega_p \rightarrow\mathbb{R}$ such that \eqref{eq::elasticity}--\eqref{eq::poroel} coupled with \eqref{eq::contstress_elporo}--\eqref{eq:noflowrate} hold. 
Finally, we close the problem by considering initial conditions $\UU(\cdot,0) = \UU_0 =  (\bm{u}_{0e}, \bm{u}_{0p},\bm{u}_{0f})$ and $\dot{\UU}(\cdot,0) = \VV_0 = (\bm{v}_{0e}, \bm{v}_{0p},\bm{v}_{0f})$.

\subsection{Weak formulation and stability estimates}
In order to state the weak formulation taking into account the essential boundary conditions, we also introduce the subspaces 
$$
\begin{aligned}
\bm{V}^i &= \{ \bm{v}\in \bm{H}^1(\Omega_i)  \,|\, \bm{v}_{|\Gamma_{i}} = \bm{0}\}, \; i =\{e,p\}, \\
\bm{W}^p &= \{ \bm{z}\in \bm{H}(\textrm{div},\Omega_p) \,|\, (\bm{z}\cdot\bm{n}_p)_{|\Gamma_{p}} = 0\}.
\end{aligned}
$$
and define the Hilbert spaces $ \bm V_0 = \bm L^2(\Omega_e) \times \bm L^2(\Omega_p) \times \bm L^2(\Omega_p)$
and $\bm V = \bm V^e \times \bm V^p \times \bm{W}^p$, endowed with the following norm and seminorm, respectively
\begin{align*}
\norm{\UU}^2_{0} & = \norm{\rho_e^{1/2}\bm u_e}_{\Omega_e}^2 + \norm{\rho_u^{1/2}\bm  u_p}_{\Omega_p}^2 
    +  \norm{(\rho_f \phi)^{1/2}(\bm u_p + \phi^{-1} \bm u_f )}_{\Omega_p}^2 \quad \forall \, \UU \in \bm V_0, \\   
    |\UU|^2_{\bm V} & = \norm{\mathbb{C}^{1/2}:\bm{\epsilon}(\bm u_e)}_{\Omega_e}^2 + \norm{\rho_e^{1/2} \zeta \bm u_e}_{\Omega_e}^2 + \norm{\mathbb{C}^{1/2}:\bm{\epsilon}(\bm u_p)}_{\Omega_p}^2 + \norm{\rho^{1/2} \zeta \bm u_p}_{\Omega_p}^2  \\
    & \quad + \norm{m^{1/2}\nabla\cdot(\beta \bm u_p + \bm u_f)}_{\Omega_p}^2 \quad \forall \, \UU \in \bm V, 
\end{align*}
where $\rho_u = (1-\phi)\rho_s/2.$
The weak form  of \eqref{eq::elasticity}--\eqref{eq:noflowrate} reads as: 
for any $t\in(0,T]$, find $ \UU \in \bm V $ s.t. $\bm u_e = \bm u_p$ on $\Gamma_I$ and 
\begin{equation}\label{eq::weakform}
\mathcal{M} (\ddot{\UU}, \VV) + \mathcal{D}(\dot{\UU},\VV) +
\mathcal{A}(\UU,\VV)  + \mathcal{C}(\UU,\VV) =
\mathcal{F}(\VV) \quad  \forall \; \VV \in \bm V,
\end{equation}
 with $\UU(\cdot,0) = \UU_0 \in \bm V$ and $\dot{\UU}(\cdot,0)= \VV_0 \in \bm V_0$.
In \eqref{eq::weakform} $\mathcal{F} : \Omega_e \times \Omega_p \times \Omega_p \rightarrow \mathbb{R} $ is a linear functional defined as $$ \mathcal{F}(\VV) =  (\bm{f}_e, \bm v_e)_{\Omega_e} +  (\bm{f}_p,\bm v_p)_{\Omega_p} + (\bm{g}_p,\bm v_f)_{\Omega_p}
$$
for any $ \VV = (\bm{v}_e,\bm{v}_p,\bm{v}_f) \in \bm V$, while for any $\UU, \VV \in \bm V$ we have set 
\begin{equation}\label{eq:bilinear_forms}
    \begin{aligned}
    \mathcal{M}(\UU,\VV)  & = (\rho_e \bm u_e, \bm v_e)_{\Omega_e} +  (\rho_p \bm{u}_p + \rho_f \bm u_f , \bm{v}_p)_{\Omega_p} 
    +  ( \rho_f \bm u_p + \rho_w \bm u_f, \bm v_f)_{\Omega_p}, \\
    \mathcal{D} (\UU,\VV) & = 
	 (2\rho_e \zeta \bm u_e, \bm v_e)_{\Omega_e} + (2\rho_p \zeta \bm u_p, \bm v_p)_{\Omega_p}  + (\eta k^{-1} \bm{u}_f,\bm{v}_f)_{\Omega_p}, \\
    \mathcal{A}(\UU,\VV)  & =    (\bm{\sigma}_e(\bm{u}_e),\bm{\epsilon}(\bm{v}_e))_{\Omega_e}+ (\rho_e \zeta^2 \bm u_e, \bm v_e)_{\Omega_e} + (\bm{\sigma}_e(\bm{u}_p),\bm{\epsilon}(\bm{v}_p))_{\Omega_p} + (\rho_p \zeta^2 \bm u_p, \bm v_p)_{\Omega_p} \\ & \quad + 
    (m\nabla\cdot(\beta\bm{u}_p+\bm u_f),\nabla\cdot(\beta\bm{v}_p+\bm{v}_f))_{\Omega_p} \\
\mathcal{C}(\UU,\VV) & = -\langle \bm{\sigma}_e(\bm u_e)\bm{n}_p,\bm{v}_p - \bm v_e \rangle_{\Gamma_I} + 
\langle p(\bm{u}_p,\bm u_f), ((1-\delta)\beta \bm{v}_p + \bm{v}_f)\cdot\bm n_p \rangle_{\Gamma_I}, 
\end{aligned}
\end{equation}
with $\langle \cdot,\cdot\rangle_{\Gamma_I}$  in \eqref{eq:bilinear_forms} denoting the $H^{\frac12}(\Gamma_I)$-$H^{-\frac12}(\Gamma_I)$ duality product. Before presenting a stability estimate for the solution of problem \eqref{eq::weakform} we define, for all $\UU \in C^1([0,T];\bm V_0)\cap C^0([0,T];\bm V)$, the energy norm
\begin{equation}\label{eq:energy}
\|\UU(t)\|_{\mathcal{E}}^2 =
\norm{\dot{\UU}(t)}^2_{0} + |\UU(t)|_{\bm V}^2 +  \int_0^t \mathcal{D}(\dot{\UU},\dot{\UU})(s) \,ds + \mathcal{D}(\UU,\UU)(0) \quad t \in (0,T],
\end{equation}
and we adopt the notation  $\|\UU_0\|_{\mathcal{E}}^2 = \|\UU(0)\|_{\mathcal{E}}^2 =
\norm{\VV_0}^2_{0} + |\UU_0|_{\bm V}^2 + \mathcal{D}(\UU_0,\UU_0)$.
As a result of Lemma \ref{lemma:stab} below,  $\max_{0\leq t \leq T }\norm{\cdot}_{\mathcal{E}}$ is a norm on $C^1([0,T];\bm V_0)\cap C^0([0,T];\bm V)$.

\begin{lemma}\label{lemma:stab}
The bilinear forms $\mathcal{M}$, $\mathcal{D}$, and $\mathcal{A}$ defined in \eqref{eq:bilinear_forms} are such that
\begin{align}
 \mathcal{M}(\UU,\VV) & \lesssim \norm{\UU}_{0} \norm{\VV}_{0} &  \forall \UU,\VV \in \bm V_0, \label{eq:M-cont}\\ 
\mathcal{M}(\UU,\UU) & \gtrsim \norm{\UU}_{0}^2 &  \forall \UU \in \bm V_0, \label{eq:M-coer} \\
 \mathcal{A}(\UU,\VV)  & \lesssim 
 |\UU |_{\bm V} |\VV|_{\bm V} &  \forall \UU,\VV \in \bm V, \label{eq:A-cont}\\
\mathcal{A}(\UU,\UU) &  = |\UU |_{\bm V}^2   
 & \forall \UU \in \bm V  \label{eq:A-coer}.
\end{align}
\end{lemma}
\begin{proof}.
Inequalities \eqref{eq:M-cont} and \eqref{eq:A-cont} are readily inferred by applying the Cauchy--Schwarz and triangle inequalities.
For \eqref{eq:M-coer} we use the definition of the density functions $\rho_e$, $\rho_p$, $\rho_u$, and $\rho_w$, we observe that  $a >1$ to prove that $\mathcal{M}(\UU,\UU) \gtrsim \norm{\UU}_0^2$, cf. \cite{ABM_Vietnam}.
The last equality \eqref{eq:A-coer} follows by the definition \eqref{eq:bilinear_forms}.
\end{proof}
\begin{theorem}[Stability of the continuous weak formulation]\label{thm:stability}
Assume that the problem data satisfy $\mathsf{F} = (\bm f_e, \bm f_p, \bm g_p)\in L^2((0,T);\bm V_0)$, $\UU_0 \in\bm V$, and $\VV_0 \in \bm V_0$. For any $t\in(0,T]$, let $\UU=(\bm u_e,\bm u_p, \bm u_f) \in \bm V$ be the solution of \eqref{eq::weakform}. Then, it holds
\begin{align*}
 \max_{0\leq t \leq T}\norm{\UU(t)}^2_{\mathcal{E}}\lesssim \norm{\UU_0}^2_{\mathcal{E}}
+ \int_0^T & \Big(\norm{(2\rho_e\zeta)^{-1/2}\bm f_e}_{\Omega_e}^2  +   \norm{(2\rho_e\zeta)^{-1/2}\bm  f_p}_{\Omega_p}^2  \\ &
    +  \norm{(\eta/k)^{-1/2} \bm g_p}_{\Omega_p}^2\Big) ds,
\end{align*}
with hidden constant depending on the observation time $t$ and the material properties.
\end{theorem}
\begin{proof}.
The proof follows the lines of the one proposed in \cite[Proposition 5.1]{ABM_Vietnam}. We take $\dot{\UU} = (\dot{\bm u}_e,\dot{\bm u}_p, \dot{\bm u}_f)$ as test functions in \eqref{eq::weakform} so that the interface terms are null, i.e.,  $\mathcal{C}(\UU,\dot{\UU}) = 0$, thanks to condition \eqref{eq::cont_disp}. We integrate in time between $0$ and $t$, add on both side $\mathcal{D}(\UU_0,\UU_0)\ge 0$, and  use \eqref{eq:M-cont}--\eqref{eq:A-coer} to get 
\begin{equation*}
\norm{\UU(t)}^2_{\mathcal{E}} 
+ \int_0^t \mathcal{D}(\dot{\UU},\dot{\UU})(s) \,ds 
\lesssim \norm{\UU_0}^2_{\mathcal{E}} 
+ 2\int_0^t \mathcal{F}(\dot{\UU})(s) ds.    
\end{equation*}
The thesis follows by applying the Cauchy-Schwarz and Young inequalities to bound the second term on the right-hand side. 
\end{proof}

\section{Semi-discrete formulation and stability analysis}\label{sec::numerical}

We introduce a \textit{polytopic} mesh $\mathcal{T}_h$ made of general polygons (in 2d) or polyhedra (in 3d) and write $\mathcal{T}_h$ as $\mathcal{T}_h=\mathcal{T}^e_h \cup\mathcal{T}^p_h$, where $\mathcal{T}^{\#}_h=\{\kappa\in\mathcal{T}_h:\overline{\kappa}\subseteq\overline{\Omega}_{\#}\}$, with $\#=\{e,p\}$.
Implicit in this decomposition there is the assumption that the meshes $\mathcal{T}_h^e$  and $\mathcal{T}_h^p$ are aligned with $\Omega_e$ and $\Omega_p$, respectively. 
Polynomial degrees $p_{e,\kappa}\geq1$ and  $p_{p,\kappa}\geq1$ are associated with each element of $\mathcal{T}_h^e$ and $\mathcal{T}_h^p$. 
The discrete spaces are introduced as follows: $\bm{V}_h^e=[\mathcal{P}_{p_e}(\mathcal{T}_h^e)]^d$ and $\bm{V}_h^p=[\mathcal{P}_{p_p}(\mathcal{T}_h^p)]^d$ where $\mathcal{P}_{r}(\mathcal{T}^{\#}_h)$ is the space of piecewise polynomials in $\Omega_{\#}$ of total degree less than or equal to $r_\kappa$ in any $\kappa\in\mathcal{T}_h^{\#}$ with $\#=\{e,p\}$.

In the following, we assume that $\mathbb{C}$ and $m$ are element-wise constant and we define 
$\overline{\mathbb{C}}_\kappa=(|\mathbb{C}^{1/2}|_2^2)_{|\kappa}$ for all $\kappa\in\mathcal{T}_h^e \cup \mathcal{T}_h^p$, and $\overline{m}_{\kappa}=( m)_{|\kappa}$ for all $\kappa\in\mathcal{T}_h^p$.
The symbol $|\cdot|_2$ stands for \MB{the $\ell^2$-norm on $\mathbb{R}^{n\times n}$, with $n=3$ if $d=2$ and $n=6$ if $d=3$.}
In order to deal with polygonal and polyhedral elements, we define an \textit{interface}  as the intersection of the $(d-1)$-dimensional faces of  
any two neighboring elements of $\mathcal{T}_h$. If $d=2$, an interface/face is a line segment and the set of all interfaces/faces is  denoted by $\mathcal{F}_h$.
When $d=3$, an interface can be a general polygon that we assume could be further decomposed into a set of planar triangles collected in the set  $\mathcal{F}_h$.  
We decompose $\mathcal{F}_h$ as $\mathcal{F}_h=\mathcal{F}_{h}^I \cup \mathcal{F}_h^e \cup  \mathcal{F}_h^p$, where
$ \mathcal{F}_{h}^I = \{F\in\mathcal{F}_h:F\subset\partial \kappa^e\cap\partial \kappa^p,\kappa^e\in\mathcal{T}_{h}^e,\kappa^p\in\mathcal{T}_{h}^p\},$  and 
$\mathcal{F}_h^e$, and $\mathcal{F}_h^p$ denote all the faces of $\mathcal{T}_h^e$, and $\mathcal{T}_h^p$ respectively, not laying on $\Gamma_I$.
Finally, the faces of $\mathcal{T}_h^e$ and $\mathcal{T}_h^p$ can be further written as the union of \textit{internal} ($i$) and \textit{boundary} ($b$) faces, respectively, i.e.:
$\mathcal{F}^e_h=\mathcal{F}^{e,i}_h\cup\mathcal{F}^{e,0}_h$ and $\mathcal{F}^p_h=\mathcal{F}^{p,i}_h\cup\mathcal{F}^{p,0}_h$.

Following \cite{CangianiDongGeorgoulisHouston_2017}, we next introduce the main assumption on $\mathcal{T}_h$. 
\begin{definition}\label{def::polytopic_regular}
A mesh $\mathcal{T}_h$ is said to be \textit{polytopic-regular} if for any $ \kappa \in \mathcal{T}_h$, there exists a set of non-overlapping $d$-dimensional simplices contained in $\kappa$, denoted by $\{S_\kappa^F\}_{F\subset{\partial \kappa}}$, such that for any face $F\subset\partial \kappa$, it holds $h_\kappa\lesssim d|S_\kappa^F| \, |F|^{-1}$.
\end{definition}
\begin{assumption}
The sequence of meshes $\{\mathcal{T}_h\}_h$ is assumed to be \textit{uniformly} polytopic regular in the sense of  \eqref{def::polytopic_regular}.
\label{ass::regular}
\end{assumption}
\noindent
As pointed out in \cite{CangianiDongGeorgoulisHouston_2017}, this assumption does not impose any restriction on either the number of faces per element or their measure relative to the diameter of the element they belong to.
Under \eqref{ass::regular}, the following \textit{trace-inverse inequality} holds:
\begin{align}
& ||v||_{L^2(\partial \kappa)}\lesssim ph_\kappa^{-1/2}||v||_{L^2(\kappa)}
&& \forall \ \kappa\in\mathcal{T}_h \ \forall v \in \mathcal{P}_p(\kappa).
\label{eq::traceinv}
\end{align}
In order to avoid technicalities, we also assume that $\mathcal{T}_h$ satisfies a \textit{hp-local bounded variation} property:
\begin{assumption}
For any pair of neighboring elements $\kappa^\pm\in\mathcal{T}_h^{\#}$, it holds $h_{\kappa^+}\lesssim h_{\kappa^-}\lesssim h_{\kappa^+},\ \ p_{\#,\kappa^+}\lesssim p_{\#,\kappa^-}\lesssim p_{\#,\kappa^+}$, with $\#=\{e,p\}$.
\label{ass::3}
\end{assumption}

Finally, following  \cite{Arnoldbrezzicockburnmarini2002}, for sufficiently piecewise smooth scalar-, vector- and tensor-valued fields $\psi$, $\bm{v}$ and $\bm{\tau}$, respectively, we define the averages and jumps on each \textit{interior} face $F\in\mathcal{F}_h^{e,i}\cup\mathcal{F}_h^{p,i}\cup  \mathcal{F}_h^{I}$ shared by the elements $\kappa^{\pm}\in \mathcal{T}_h$ as follows:
\begin{align*}
 \llbracket\psi\rrbracket  &=  \psi^+\bm{n}^++\psi^-\bm{n}^-, 
 &&\llbracket\bm{v}\rrbracket  = \bm{v}^+\otimes\bm{n}^++\bm{v}^-\otimes\bm{n}^-, 
 &&\MB{\llbracket\bm{v}\rrbracket_{\bm n}  = \bm{v}^+\cdot\bm{n}^++\bm{v}^-\cdot\bm{n}^-}, \\
 \llbrace\psi \rrbrace &= \frac{\psi^++\psi^-}{2}, 
&&\llbrace\bm{v}\rrbrace   = \frac{\bm{v}^++\bm{v}^-}{2}, 
&&\llbrace\bm{\tau}\rrbrace  = \frac{\bm{\tau}^++\bm{\tau}^-}{2}, 
\end{align*}
where $\otimes$ is the tensor product in $\mathbb{R}^3$, $\cdot^{\pm}$ denotes the trace on $F$ taken within $\kappa^\pm$, and $\bm{n}^\pm$ is the outer normal vector to $\partial \kappa^\pm$. Accordingly, on \textit{boundary} faces $F\in\mathcal{F}_h^{e,0}\cup\mathcal{F}_h^{p,0}$, we set
$
\llbracket\psi\rrbracket = \psi\bm{n},\
\llbrace\psi \rrbrace = \psi,\
\llbracket\bm{v}\rrbracket= \bm{v}\otimes\bm{n},\
\MB{\llbracket\bm{v}\rrbracket_{\bm n} =\bm{v}\cdot\bm{n}},\
\llbrace\bm{v}\rrbrace= \bm{v},\
\llbrace\bm{\tau}\rrbrace= \bm{\tau}.$

For later use, we also define $\nabla_h$ and $\nabla_h \cdot$ to be the broken gradient and divergence operators, respectively, 
set $\bm{\epsilon}_h(\bm{v})=\frac{\nabla_h \bm{v} + \nabla_h \bm{v}^T}{2}$, $\bm{\sigma}_{eh}(\bm{v})=\mathbb{C}:\bm{\epsilon}_h(\bm{v})$, and use the short-hand notation  
$(\cdot,\cdot)_{\Omega_{\#}}=\sum_{\kappa\in\mathcal{T}_h^{\#}} \int_\kappa\cdot$ and $\langle\cdot,\cdot\rangle_{\mathcal{F}_h^{\#}}=\sum_{F\in\mathcal{F}_h^{\#}}\int_F\cdot$ for $\# = \{e,p\}$.

\subsection{Semi-discrete PolydG formulation}\label{sec:PolydG_form}
We define the space $\bm V_h = \bm V_h^e \times \bm V_h^p \times \bm V_h^p$  and denote by $\UU_h = (\bm u_e,\bm u_p,\bm u_f)_h$ a generic function  in $\bm V_h$. The semi-discrete PolydG formulation of problem \eqref{eq::weakform} reads as: find $\UU_h \in \bm V_h$ such that 
\begin{equation}
\mathcal{M}(\ddot{\UU}_h,\VV_h) +
\mathcal{D} (\dot{\UU}_h,\VV_h) +
\mathcal{A}_h(\UU_h,\VV_h) +
\mathcal{C}_{h}(\UU_h,\VV_h) = 
\mathcal{F}(\VV_h) \quad \forall \VV_h  \in \bm V_h.
\label{eq::dgsystem_0}
\end{equation}
 As initial conditions we take suitable projections onto $\bm V_h$ of the initial data, namely $\UU_h(0) = \UU_{0h}$  and $\dot{\UU}_h(0) = \VV_{0h}$. 
The bilinear forms $\mathcal{M}$ and $\mathcal{D}$ appearing in \eqref{eq::dgsystem_0} are defined as in \eqref{eq:bilinear_forms} while 
\begin{equation}
\mathcal{A}_h(\UU,\VV)   = \mathcal{A}_h^e(\bm{u}_e,\bm{v}_e) + \mathcal{A}_h^p(\bm{u}_p,\bm{v}_p) +
\mathcal{B}_h^p(\beta\bm{u}_p+\bm u_f,\beta\bm{v}_p+\bm{v}_f)      \label{eq:bilinear_Ah}
\end{equation}
for all $\UU = (\bm{u}_e,\bm{u}_p,\bm{u}_f)$ and $\VV = (\bm v_e, \bm v_p, \bm{v}_f) \in \bm{V}_h$ 
with 
\begin{equation}\label{eq::DGbilinearforms}
\begin{aligned}
\mathcal{A}_h^\star(\bm{u}_\star,\bm{v}_\star)
& = (\bm{\sigma}_{eh}(\bm{u}_\star),\bm{\epsilon}_h(\bm{v}_\star))_{\Omega_\star} + (\rho_\star \zeta^2 \bm u_\star, \bm v_\star)_{\Omega_\star}  - \langle\llbrace\bm{\sigma}_{eh}(\bm{u}_\star)\rrbrace,\llbracket\bm{v}_\star\rrbracket
\rangle_{\mathcal{F}_{h}^\star} 
  - \langle\llbracket\bm{u}_\star\rrbracket, \llbrace\bm{\sigma}_{eh}(\bm{v}_\star)\rrbrace \rangle_{\mathcal{F}_{h}^\star} \\ & \quad + 
\langle\alpha\llbracket\bm{u}_\star\rrbracket,\llbracket\bm{v}_\star\rrbracket\rangle_{\mathcal{F}_{h}^\star}, \quad \star = \{e,p\}, \\
\mathcal{B}_h^p(\bm{u},\bm{v})&=
( m\nabla_h\cdot\bm{u},\nabla_h\cdot\bm{v})_{{\Omega}}
-\langle\llbrace m(\nabla_h\cdot\bm{u})\rrbrace, \llbracket\bm{v}\rrbracket_{\bm n}\rangle_{\mathcal{F}_{h}^p}  -
\langle\llbracket\bm{u}\rrbracket_{\bm n},\llbrace m(\nabla_h\cdot\bm{v})\rrbrace\rangle_{\mathcal{F}_{h}^p}  + \langle \gamma \llbracket\bm{u}\rrbracket_{\bm n}, \llbracket\bm{v}\rrbracket_{\bm n}\rangle_{\mathcal{F}_{h}^p},
\end{aligned}
\end{equation}
cf also \cite{AntoniettiMazzieriNatipoltri2021}. 
Here, the stabilization functions $\alpha\in L^\infty(\mathcal{F}_h^\star)$, for $\star = \{p,e\}$ and $\gamma\in L^\infty(\mathcal{F}_h^p)$ are defined s.t.
\begin{align}
&\alpha|_F=
\begin{cases}
c_1 \max\limits_{\kappa\in\{\kappa^+,\kappa^-\}}\left(\overline{\mathbb{C}}_\kappa\ p_{\star,\kappa}^2{h_\kappa^{-1}}\right) \hspace{7,5mm}&\forall F\in\mathcal{F}_h^{\star,i}\cup \mathcal{F}_h^{I},\hspace{0,5mm}F\subseteq\partial \kappa^+\cap\partial \kappa^-, \vspace{0.1cm}\label{eq::stab_1}\\
\overline{\mathbb{C}}_\kappa \ p_{\star,\kappa}^2{h_\kappa^{-1}}&\forall F\in\mathcal{F}_h^{\star,b},\hspace{8,5mm}F\subseteq\partial \kappa,
\end{cases}\\ \nonumber\\
&
\gamma|_F=
\begin{cases}
c_2 \max\limits_{\kappa\in\{\kappa^+,\kappa^-\}}\left(\overline{ m}_\kappa \ p_{p,\kappa}^2{h_\kappa^{-1}}\right) \hspace{7,5mm}&\forall F\in\mathcal{F}_h^{p,i},\hspace{8,5mm}F\subseteq\partial \kappa^+\cap\partial \kappa^-, \vspace{0.1cm}\label{eq::stab_2}\\
\overline{ m}_\kappa \ p_{p,\kappa}^2{h_\kappa^{-1}}&\forall F\in\MB{\mathcal{F}_h^{p,b}\cup\mathcal{F}_h^{I}},\hspace{0,5mm} F\subseteq\partial \kappa, \kappa\in\mathcal{T}_h^p,
\end{cases}
\end{align}
with positive constants $c_1$ and $c_2$ that have to be properly chosen.
The definition of the penalty functions \eqref{eq::stab_1}--\eqref{eq::stab_2} is based on \cite[Lemma 35]{CangianiDongGeorgoulisHouston_2017}.
Alternative stabilization functions can be defined in the spirit of \cite{M2AN_2013__47_3_903_0}. 
The analysis of the latter is however beyond the scope of this work. 
The bilinear form $\mathcal{C}_{h}(\cdot,\cdot)$ is responsible for  the coupling between the elastic and the  poroelastic domain and is defined as the sum of five contributions, namely
\begin{multline*}
 \mathcal{C}_{h}(\UU,\VV)  =  \mathcal{A}^{e}_{\Gamma_I}( \bm u_{e},\bm v_{e})  + \mathcal{A}^{p}_{\Gamma_I}( \bm u_{p},\bm v_{p})  + \mathcal{B}^{pp}_{\Gamma_I}( \bm{u}_{p},\bm v_{p})  
+ \mathcal{B}^{pf}_{\Gamma_I}( \bm{u}_{p},\bm{v}_{f})  \\
+ \mathcal{B}^{fp}_{\Gamma_I}( \bm u_{f},\bm v_{p}) 
+ \mathcal{B}^{ff}_{\Gamma_I}(\bm u_{f}, \bm{v}_{f})
+ \mathcal{C}^{ep}_{\Gamma_I}(\bm u_{e},\bm v_{p}) +\mathcal{C}^{pe}_{\Gamma_I}( \bm u_{p},\bm v_{e}),
\end{multline*}
for any $\UU,\VV \in \bm V_h$, where
\begin{align}
    \mathcal{A}^{e}_{\Gamma_I}(\bm u_e,\bm v_e)  = & \langle  \bm \sigma_{eh} (\bm u_e)  \bm n_p,  \bm v_e \rangle_{\mathcal{F}_h^{I}} + \langle  \bm \sigma_{eh} (\bm v_e)  \bm n_p,  \bm u_e \rangle_{\mathcal{F}_h^{I}}  +  \langle \alpha \bm u_e, \bm v_e \rangle_{\mathcal{F}_h^{I}}, \label{def::bilinear_gammai} \\
    \mathcal{A}^{p}_{\Gamma_I}(\bm u_p,\bm v_p)  = &  \langle \alpha \bm u_p, \bm v_p \rangle_{\mathcal{F}_h^{I}}, \nonumber \\
\mathcal{B}^{pp}_{\Gamma_I}(\bm u_p,\bm v_p) = 
& - \langle  m \beta \nabla_h \cdot \bm u_p,  (1-\delta) \beta \bm v_p \cdot {\bm n}_p \rangle_{\mathcal{F}_h^{I}} - \langle  m \beta \nabla_h \cdot \bm v_p,  (1-\delta) \beta \bm u_p \cdot {\bm n}_p \rangle_{\mathcal{F}_h^{I}}
\nonumber \\ & \qquad + \langle \gamma  (1-\delta) \beta \bm u_p \cdot {\bm n}_p, (1-\delta) \beta \bm v_p \cdot {\bm n}_p \rangle_{\mathcal{F}_h^{I}}, \nonumber \\
\mathcal{B}^{pf}_{\Gamma_I}(\bm u_p,\bm v_f) = & 
- \langle  m \beta \nabla_h \cdot \bm u_p,   \bm v_f \cdot {\bm n}_p \rangle_{\mathcal{F}_h^{I}} - \langle  m \nabla_h \cdot \bm v_f,  (1-\delta) \beta \bm u_p \cdot {\bm n}_p \rangle_{\mathcal{F}_h^{I}}  
 + \langle \gamma (1-\delta)\beta \bm u_p \cdot {\bm n}_p, \bm v_f \cdot {\bm n}_p \rangle_{\mathcal{F}_h^{I}}, \nonumber \\
\mathcal{B}^{fp}_{\Gamma_I}(\bm u_f,\bm v_p) = & 
- \langle  m  \nabla_h \cdot \bm u_f,  (1-\delta)\beta \bm v_p \cdot {\bm n}_p \rangle_{\mathcal{F}_h^{I}} - \langle  m\beta \nabla_h \cdot \bm v_p,  \bm u_f \cdot {\bm n}_p \rangle_{\mathcal{F}_h^{I}}   
+ \langle \gamma  \bm u_f \cdot {\bm n}_p, (1-\delta)\beta \bm v_p \cdot {\bm n}_p \rangle_{\mathcal{F}_h^{I}}, \nonumber \\
\mathcal{B}^{ff}_{\Gamma_I}(\bm u_f,\bm v_f) = & - \langle  m  \nabla_h \cdot \bm u_f,  \bm v_f \cdot {\bm n}_p \rangle_{\mathcal{F}_h^{I}} - \langle  m \nabla_h \cdot \bm v_f,  \bm u_f \cdot {\bm n}_p \rangle_{\mathcal{F}_h^{I}} + \langle \gamma  \bm u_f \cdot {\bm n}_p, \bm v_f \cdot {\bm n}_p \rangle_{\mathcal{F}_h^{I}}, \nonumber \\
  \mathcal{C}^{ep}_{\Gamma_I}(\bm u_e,\bm v_p)  = &   - \langle  \bm \sigma_{eh} (\bm u_{e})  \bm n_p,  \bm v_{p}  \rangle_{\mathcal{F}_h^{I}} -\langle \alpha \bm u_{e}, \bm v_{p} \rangle_{\mathcal{F}_h^{I}}, \nonumber \\
 \mathcal{C}^{pe}_{\Gamma_I}(\bm u_p,\bm v_e)  = & - \langle   \bm  \sigma_{eh} (\bm v_{e})  \bm n_p,  \bm u_{p} \rangle_{\mathcal{F}_h^{I}} -\langle \alpha \bm u_{p}, \bm v_{e} \rangle_{\mathcal{F}_h^{I}}.\nonumber 
\end{align}
The derivation of the coupling bilinear form $\mathcal{C}_h(\cdot,\cdot)$, starting from the strong formulation of the poroelastic-elastic problem, is detailed in the Appendix. 

\begin{remark}
Notice that the coupling conditions \eqref{eq::contstress_elporo} and \eqref{eq::cont_disp} are imposed (weakly) through the bilinear forms $\mathcal{A}^{e}_{\Gamma_I}(\cdot,\cdot),\mathcal{A}^{p}_{\Gamma_I}(\cdot,\cdot),\mathcal{C}^{ep}_{\Gamma_I}(\cdot,\cdot)$ and $\mathcal{C}^{pe}_{\Gamma_I}(\cdot,\cdot)$, while condition \eqref{eq:noflowrate} is included in the bilinear forms $\mathcal{B}^{pp}_{\Gamma_I}(\cdot,\cdot)$, $\mathcal{B}^{pf}_{\Gamma_I}(\cdot,\cdot)$, $\mathcal{B}^{fp}_{\Gamma_I}(\cdot,\cdot)$, and $\mathcal{B}^{ff}_{\Gamma_I}(\cdot,\cdot)$. The latter couples at the interface  the filtration displacement in $\Omega_p$ to the elastic displacement in $\Omega_e$, through the elastic displacement in $\Omega_p$.
\end{remark}

By fixing a basis for the space $\bm V_h$ and denoting by $\bm{\UU}_h =(\bm \UU_e, \bm \UU_p, \bm \UU_f)^T \in \mathbb{R}^{\ndof}$ the vector of the $\ndof$ expansion coefficients in the chosen basis of the unknown $\UU_h$, the semi-discrete formulation \eqref{eq::dgsystem_0} can be written equivalently as:
\begin{multline}\label{eq::algebraic}
\left[ \begin{matrix} 
M^e_{\rho_e} & 0 & 0 \\
0 & M_\rho^p & M_{\rho_f}^p  \\
0 & M_{\rho_f}^p & M_{\rho_w}^p  \\
\end{matrix} \right]  
\left[ \begin{matrix} 
\ddot{\bm \UU}_e \\
\ddot{\bm \UU}_p \\
\ddot{\bm \UU}_f 
\end{matrix} \right] + 
\left[ \begin{matrix} 
 D^e_{\rho_e} & 0 & 0 \\
0 &  D^p_{\rho} & 0 \\
0 & 0 & D^f
\end{matrix} \right]  
\left[ \begin{matrix} 
\dot{\bm \UU}_e \\
\dot{\bm \UU}_p \\
\dot{\bm \UU}_f 
\end{matrix} \right] \\
+\left[ \MB{\begin{matrix} 
 A^e_h +{A}_{\Gamma_I}^e &  C_{\Gamma_I}^{pe} & 0 \\
 C_{\Gamma_I}^{ep} & {A}^p + {A}_{\Gamma_I}^p  +  B_{\beta^2}^p  + B_{\Gamma_I}^{pp}&  {B}_\beta^p + B_{\Gamma_I}^{fp} \\
0 & {B}_\beta^p + B_{\Gamma_I}^{pf} & {B}^p + B_{\Gamma_I}^{ff} 
\end{matrix}} \right] 
\left[ \begin{matrix} 
{\bm \UU}_e \\
{\bm \UU}_p \\
{\bm \UU}_f 
\end{matrix} \right] = 
\left[ \begin{matrix} 
\bm {\FF}^e \\
\bm {\FF}^p \\ 
\bm {\GG}^p
\end{matrix} \right] 
\end{multline}
with initial conditions $ \bm{\UU}_{0h}$ and $\bm{\VV}_{0h}$.
We remark that $\bm{\FF}^e$, $\bm{\FF}^p$ and $\bm{\GG}^p$ are the vector representations of the linear functional 
$\mathcal F$. 
To ease the notation, we set $\bm{\FF}_h = [ \bm {\FF}^e, \bm {\FF}^p,  \bm {\GG}^p]^T$ and we rewrite system  \eqref{eq::algebraic} in compact form as
\begin{equation}\label{eq::algebraic_compact}
{M}\ddot{\bm{\UU}}_h(t)+{D}\dot{\bm{\UU}}_h(t)+ (A+B+C)\bm{\UU}_h(t)=\bm{\FF}_h(t)
\quad\forall t\; \in (0,T].
\end{equation}

\subsection{Stability analysis}
To carry out the stability analysis of the semi-discrete problem, we introduce the energy norm
\begin{equation}
\label{eq::norm_def1}
\norm{\UU(t)}_{\rm E}^2=
\norm{\dot{\UU}(t)}_{0}^2 +
|\UU(t)|_{\rm dG}^2 + |\UU(t)|^2_{\Gamma_I} + \int_0^t \mathcal{D}(\dot{\UU},\dot{\UU})(s) \,ds + \mathcal{D}(\UU,\UU)(0),
\end{equation}
for all  $\UU = (\bm u_e,\bm u_p,\bm u_f) \in C^1([0,T];\bm V_h)$, where
\begin{align*}
|\UU|_{\rm dG}^2 & = 
\norm{\bm u_e(t)}_{\rm dG,e}^2 + \norm{\bm u_p(t)}_{\rm dG,p}^2 + |(\beta\bm u_p+\bm u_f)(t)|_{\rm dG,p}^2, \\
\norm{\bm{u}_\star}_{\rm dG,\star}^2 & =\norm{\mathbb{C}^{1/2}:\bm{\epsilon}_h(\bm{u}_\star)}_{\Omega_\star}^2 + \norm{\alpha^{1/2}\llbracket\bm{u}_\star\rrbracket}_{\mathcal{F}_{h}^{\star}}^2   & \forall \bm u_\star \in \bm V_h^\star \; \star = \{e,p\},    \\
|\bm{u}_p|_{\rm dG,p}^2 &=\norm{{m}^{1/2}\nabla_h\cdot\bm u_p}_{\Omega_p}^2+
\norm{\gamma^{1/2}\llbracket\bm{u}_p\rrbracket_{\bm n}}_{\mathcal{F}_{h}^{p} }^2 & \forall \bm u_p\in \bm V_h^p,
\end{align*}
and for any $(\bm u_{e}, \bm u_{p}, \bm u_{f}) \in \bm V_h^e \times \bm V_h^p \times \bm V_h^p$ we have

\begin{align*}
|\UU|^2_{\Gamma_I} & = \norm{\alpha^{1/2}(\bm u_p - \bm u_e) }_{\mathcal{F}_h^I}^2 +\norm{\gamma^{1/2} ((1-\delta)\beta\bm u_p+\bm u_f) \cdot \bm n_p}_{\mathcal{F}_h^I}^2.
\end{align*}

%
%
Before showing the main result of the section, we introduce  the following fundamental lemmas. 
\begin{lemma}\label{lemma_1_fund}
For any $\UU,\VV \in \bm{V}_h$ and for large enough parameters $c_1,c_2$, it holds
\begin{align*}
    &\mathcal{A}_h(\UU,\UU)  \gtrsim |\UU|^2_{\rm dG}, \qquad
    &&\mathcal{A}_h(\UU,\VV)  \lesssim |\UU|_{\rm dG}|\VV|_{\rm dG}. 
\end{align*}
\end{lemma}
\begin{proof}.
For the proof we use the definition of the bilinear form $\mathcal{A}_h(\cdot,\cdot)$ in \eqref{eq:bilinear_Ah} and  combine the results in \cite[Lemma A.3]{AntoniettiMazzieriNatipoltri2021}.
\end{proof}

\begin{lemma}\label{lemma_2_fund}
For any $(\bm u_{e}, \bm u_{p}) \in \bm V_h^e \times \bm V_h^p$ and any $ \bm u_{f} \in \bm V_h^p  $ it holds
\begin{equation}
2 | \langle \bm \sigma_{eh}(\bm u_e) \bm n_e, \bm u_p - \bm u_e  \rangle_{\mathcal{F}_h^I} | \lesssim \frac{1}{\sqrt{c_1}} \norm{\bm u_{e}}_{\rm dG,e} \norm{\alpha^{1/2}(\bm{u}_p -\bm u_e)}_{\mathcal{F}_h^I},  
\end{equation}
\begin{multline}
   2| \langle m \nabla \cdot (\beta  \bm u_{p} +   \bm u_{f}), ((1-\delta) \beta \bm u_{p} + \bm u_f) \cdot \bm n_p \rangle_{\mathcal{F}_h^I} | \\ \lesssim \frac{1}{\sqrt{c_2}}  |(\beta\bm u_{p}+\bm u_{f})(t)|_{\rm dG,p}\norm{\gamma^{1/2} ((1-\delta)\beta\bm u_p+\bm u_f) \cdot \bm n_p}_{\mathcal{F}_h^I},
\end{multline}
where $c_1$ and $c_2$ are the two positive constants at our disposal appearing in the definition of the stabilization function given in  \eqref{eq::stab_1}--\eqref{eq::stab_2}.
\end{lemma}
\begin{proof}.
The proof hinges on Assumption \ref{ass::regular} and the trace inverse inequality \eqref{eq::traceinv}. See also  \cite[Lemma A.2]{AntoniettiMazzieriNatipoltri2021}
\end{proof}
\begin{corollary}\label{cor:fund_cor}
For any $\UU = (\bm u_e, \bm u_p, \bm u_f) \in  \bm V_h$ and  for $c_1$ and $c_2$ large enough it holds 
\begin{equation}\label{cont_coerc_Ah}
  |\UU|_{\rm dG}^2 + |\UU|_{\Gamma_I}^2 \lesssim  \mathcal{A}_h(\UU,\UU) +  \mathcal{C}_{h}(\UU,\UU) \lesssim  |\UU|_{\rm dG}^2 + |\UU|_{\Gamma_I}^2.
\end{equation}
\end{corollary}
\begin{proof}.
The proof follows by noting that 
\begin{multline*}
    \mathcal{C}_{h}(\UU,\WW) = - 2 \langle \bm \sigma_{e} (\bm u_e)  \bm n_p, \bm u_p -\bm u_e \rangle_{\mathcal{F}_h^{I}} 
    - 2\langle m \nabla \cdot (\beta  \bm u_{p} +   \bm u_{f}), ((1-\delta) \beta \bm u_{p} + \bm u_f) \cdot \bm n_p \rangle_{\Gamma_I}   \\ 
    +  \norm{\alpha^{1/2}(\bm u_p - \bm u_f) }_{\mathcal{F}_h^I}^2 
    +\norm{\gamma^{1/2} ((1-\delta)\beta\bm u_p+\bm u_f) \cdot \bm n_p}_{\mathcal{F}_h^I}^2,
\end{multline*}
and using  the results in Lemma~\ref{lemma_1_fund} and  Lemma~\ref{lemma_2_fund}.
\end{proof}

The main stability result is stated in the following theorem.
\begin{theorem}[Stability of the semi-discrete formulation]\label{teo:stability}
Let  \eqref{ass::regular} and \eqref{ass::3} be satisfied. For sufficiently large penalty parameters $c_1$ and $c_2$ in \eqref{eq::stab_1} and \eqref{eq::stab_2}, respectively, let $\UU_h(t) \in  \bm V_h $ be the solution of \eqref{eq::dgsystem_0} for any $t\in(0,T]$. Then, it holds
\begin{equation}\label{theo:energy_system}
\begin{aligned}
 \sup_{0\leq t \leq T}\norm{\UU_h(t)}^2_{\rm{E}}\lesssim \norm{\UU_{h0}}^2_{\rm{E}}
+ \int_0^T & \Big(\norm{(2\rho_e\zeta)^{-1/2}\bm f_e}_{\Omega_e}^2  +   \norm{(2\rho_e\zeta)^{-1/2}\bm  f_p}_{\Omega_p}^2  \\ &
    +  \norm{(k/\eta)^{1/2} \bm g_p}_{\Omega_p}^2\Big) ds,
\end{aligned}
\end{equation}
where the hidden constant depends on the final time $T$, on the penalization parameters $c_1$ and $c_2$ in \eqref{eq::stab_1}--\eqref{eq::stab_2} and on the material properties.
\end{theorem}
\begin{proof}.
By taking $\VV_h = \dot{\UU}_h \in  \bm V_h$ in \eqref{eq::dgsystem_0} we obtain
\begin{equation*}
\frac{1}{2}\frac{d}{dt}\bigg[
\mathcal{M}(\dot{\UU}_h, \dot{\UU}_h)  +
\mathcal{A}_h(\UU_h,\UU_h) + \mathcal{C}_h(\UU_h,\UU_h)   \bigg] + \mathcal{D}(\dot{\UU}_h, \dot{\UU}_h) =
\mathcal{F}(\dot{\UU}_h). 
\end{equation*}
Next, we integrate in time between $0$ and $t$, add on both side $\mathcal{D}(\UU_0,\UU_0)\geq 0$, use \eqref{cont_coerc_Ah} and \eqref{eq:M-cont}--\eqref{eq:M-coer} to obtain  
\begin{equation*}
\norm{{\UU}_h(t)}_{\rm E}^2 
+ \int_0^t \mathcal{D}(\dot{\UU}_h, \dot{\UU}_h)(s)\, ds \lesssim
 \norm{{\UU}_{h0}}_{\rm E}^2  + 2\int_0^t \mathcal{F}(\dot{\UU}_h)(s) \, ds, 
\end{equation*}
for large enough penalty constants $c_1$ and $c_2$.
Next, to estimate the last term on the right-hand side we use the Cauchy-Schwarz and Young inequalities. 
\end{proof}

\section{Semi-discrete error analysis}\label{sec:semi_discrete_error}

This section is devoted to the semi-discrete error analysis. We first recall some standard interpolation error estimates written in the context of polytopal discretizations and then present the main results of the section. We introduce the following definition and a further mesh assumption (cf. \cite{cangiani2014hp,CangianiDongGeorgoulisHouston_2017}) in order to avoid technicalities in the following proofs.
\begin{definition}\label{def:covering}
A \textit{covering} $\mathcal{T}_{\S}=\{\mathcal{K}\}$ of the polytopic mesh $\mathcal{T}_h$ is a set of regular shaped $d$-dimensional simplices $\mathcal{K}$, $d=2,3$, s.t. $\forall \ \kappa\in\mathcal{T}_h$, $\exists \ \mathcal{K}\in  \RRR{\mathcal{T}_{\S}}$ s.t. $\kappa \subseteq \mathcal{K}$.
\end{definition}

\begin{assumption}
Any mesh $\mathcal{T}_h$ admits a covering $\mathcal{T}_{\S}$ in the sense of \eqref{def:covering} such that 
i) $\max_{\kappa\in\mathcal{T}_h} \textrm{card}\{\kappa^\prime\in\mathcal{T}_h:\kappa^\prime\cap\mathcal{K}\neq\emptyset,\ \mathcal{K}\in\MB{\mathcal{T}_{\S}} \text{ s.t.} \ \kappa\subset\mathcal{K}\}\lesssim 1$ and 
ii) $h_\mathcal{K}\lesssim h_\kappa $ for each pair $\kappa\in\mathcal{T}_h, \ \mathcal{K}\in\MB{\mathcal{T}_{\S}}$ with $\kappa\subset\mathcal{K}$.
\label{ass::2}
\end{assumption}
For all  $ t\in[0,T]$ we also introduce the norm
\begin{equation*}
\trinorm{\UU(t)}_{\rm E}^2 = \norm{\dot{\UU}(t)}_{0}^2 +
\trinorm{\UU(t)}_{\rm dG}^2 + |\UU(t)|^2_{\Gamma_I} + \int_0^t \mathcal{D}(\dot{\UU},\dot{\UU})(s) \, ds
 + \mathcal{D}(\UU,\UU)(0) ,
\end{equation*}
for any $\UU = (\bm u_e, \bm u_p, \bm u_f) \in \bm H^2(\mathcal{T}_h^e) \times \bm H^2(\mathcal{T}_h^p) \times \bm H^2(\mathcal{T}_h^p)$, where 
\begin{multline}
    \trinorm{\UU}_{\rm dG}^2 = |\UU |_{\rm dG}^2 +  \norm{\alpha^{-1/2}\llbrace\mathbb{C}:\bm{\epsilon}_h(\bm{u}_e)\rrbrace}_{\mathcal{F}_h^e}^2 +
    \norm{\alpha^{-1/2}\llbrace\mathbb{C}:\bm{\epsilon}_h(\bm{u}_p)\rrbrace}_{\mathcal{F}_h^p}^2 \\
    + \norm{\gamma^{-1/2}\llbrace m \nabla_h\cdot\bm (\beta \bm u_p + \bm u_f)\rrbrace}_{\mathcal{F}_h^p}^2.
\end{multline}
Next, we prove the following lemma.
\begin{lemma} 
For any $\UU \in \bm H^2(\mathcal{T}_h^e) \times \bm H^2(\mathcal{T}_h^p) \times \bm H^2(\mathcal{T}_h^p)$ and any $\VV_h  \in  \bm{V}_h$ it holds 
\begin{align}\label{eq::continuity_extended_A}
\mathcal{A}_h(\UU,\VV_h) + \mathcal{C}_{h}(\UU,\VV_h) & \lesssim  (\trinorm{\UU}_{\rm dG} + |\UU|_{\Gamma_I}) (|\VV_h|_{\rm dG} + |\VV_h|_{\Gamma_I}).
\end{align}
\end{lemma}
\begin{proof}.
The proof is obtained by using classical dG arguments and by noting that
\begin{align*}
\mathcal{C}_{h}(\UU,\VV_h) =  &  -\langle  \bm \sigma_{eh}(\bm u_e) \bm n_p, \bm v_{ph} - \bm v_{eh}  \rangle_{\mathcal{F}_h^{I}} 
  -\langle \bm \sigma_{eh}(\bm v_{eh}) \bm n_p, \bm u_p - \bm u_e \rangle_{\mathcal{F}_h^{I}} 
  \\ & + \langle \alpha (\bm u_p - \bm u_e),  \bm v_{ph}- \bm v_{eh} \rangle_{\mathcal{F}_h^{I}} \\
  &  - \langle  m\nabla_h \cdot (\beta \bm u_{p} + \bm u_f) ,  ((1-\delta)\beta \bm v_{ph} + \bm v_{fh}) \cdot \bm n_p \rangle_{\mathcal{F}_h^{I}} \\ 
  &  - \langle  m \nabla_h \cdot (\beta \bm v_{ph} + \bm v_{fh}), ((1-\delta)\beta \bm u_p + \bm u_f) \cdot \bm n_p \rangle_{\mathcal{F}_h^{I}} \\ & +  \langle \gamma ((1-\delta)\beta \bm u_p + \bm u_f) \cdot \bm n_p, ((1-\delta)\beta \bm v_{ph} + \bm v_{fh}) \cdot \bm n_p \rangle_{\mathcal{F}_h^{I}}.
\end{align*}
See also \cite{AntoniettiMazzieri2018} and \cite[Lemma A.3]{AntoniettiMazzieriNatipoltri2021}. 
\end{proof}

For an open bounded polytopic domain $\Sigma\subset\mathbb{R}^d$ and a generic polytopic mesh $\mathcal{T}_h$ over $\Sigma$ satisfying \eqref{ass::2},
as in  \cite{cangiani2014hp}, we can introduce the Stein's extension operator $\tilde{\mathcal{E}}:H^m(\kappa)\rightarrow H^m(\mathbb{R}^d)$ \cite{stein1970singular}, for any $\kappa\in\mathcal{T}_h$ and $m\in\mathbb{N}_0$, such that $\tilde{\mathcal{E}}v|_\kappa=v$ and $\norm{\tilde{\mathcal{E}}v}_{m,\mathbb{R}^d}\lesssim\norm{v}_{m,\kappa}$. The corresponding vector-valued version mapping $\bm H^m(\kappa)$ onto $\bm H^m(\mathbb{R}^d)$ acts component-wise and is denoted in the same way. 
In what follows, for any $\kappa\in\mathcal{T}_h$, we will denote by $\mathcal{K}_\kappa$ the simplex belonging to $\mathcal{T}_{\S}$ such that $\kappa\subset\mathcal{K}_\kappa$.

The next result provides the interpolation bounds that are instrumental for the derivation of the a-priori error estimate.

\begin{lemma}\label{lemma:interp_estimate}
For any $\UU = (\bm u_e, \bm u_p, \bm u_f) \in  C^1([0,T];\bm V\cap \bm H^m(\mathcal{T}_h^e)\times\bm H^n(\mathcal{T}_h^p)\times \bm H^\ell(\mathcal{T}_h^p))$ with $m,n,\ell\geq 2$, there exists $\UU_I = (\bm u_{e}, \bm u_{p},\bm u_{f})_I \in C^1([0,T]; \bm V_h )$ such that
\begin{equation}\label{eq:interp_est}
\begin{aligned}
\max_{0\leq t\leq T} \trinorm{(\UU-\UU_I)(t)}_{\rm E }^2 \lesssim & \sum_{\kappa\in\mathcal{T}_h^e}{\frac{h_\kappa^{2(s_\kappa-1)}}{p_{e,\kappa}^{2m-3}}}\mathcal{N}(\bm u_e)^2_{m,\mathcal{K}_\kappa}
+
\sum_{\kappa\in\mathcal{T}_h^p}{\frac{h_\kappa^{2(q_\kappa-1)}}{p_{p,\kappa}^{2n-3}}}\mathcal{N}(\bm u_p)^2_{n,\mathcal{K}_\kappa} \\ & +
\sum_{\kappa\in\mathcal{T}_h^p}{\frac{h_\kappa^{2(r_\kappa-1)}}{p_{p,\kappa}^{2\ell-3}}}\mathcal{N}(\bm u_f)^2_{\ell,\mathcal{K}_\kappa},
\end{aligned}
\end{equation}
where $s_\kappa = \min(m,p_{e,\kappa}+1)$, $q_\kappa = \min(n,p_{p,\kappa}+1)$, $r_\kappa = \min(\ell,p_{p,\kappa}+1)$, and 
$$\mathcal{N}(\bm u)^2_{m,\mathcal{K}_\kappa} = \max_{0\leq t\leq T}  \left((1+T)
\norm{\widetilde{\mathcal{E}}\dot{\bm u}(t)}^2_{m,\mathcal{K}_\kappa}\hspace{-1mm}+
\norm{\widetilde{\mathcal{E}}{\bm u}(t)}^2_{m,\mathcal{K}_\kappa}\right).$$
\end{lemma}
\begin{proof}.
The first part of the proof readily follows by reasoning as in \cite[Lemma 5.1]{bonaldi}, cf. also \cite[Lemma 4.2]{AntoniettiMazzieriNatipoltri2021}. 
To infer estimate \eqref{eq:interp_est}, we resort to the $hp$-approximation properties stated in \cite[Lemma 23]{CangianiDongGeorgoulisHouston_2017}, implying 
\begin{equation*}
\|\dot{\UU}-\dot{\UU}_I \|_0^2
\lesssim  \sum_{\kappa\in\mathcal{T}_h^e} 
{\frac{h_\kappa^{2s_\kappa}}{p_{e,\kappa}^{2m}}}\norm{\widetilde{\mathcal{E}}\dot{\bm u_e}}_{m,\mathcal{K}_\kappa}^2 + 
 \sum_{\kappa\in\mathcal{T}_h^p} 
{\frac{h_\kappa^{2q_\kappa}}{p_{p,\kappa}^{2n}}} \norm{\widetilde{\mathcal{E}}\dot{\bm u_p}}_{n,\mathcal{K}_\kappa}^2 
 + 
 \sum_{\kappa\in\mathcal{T}_h^p} 
{\frac{h_\kappa^{2r_\kappa}}{p_{p,\kappa}^{2\ell}}}\norm{\widetilde{\mathcal{E}}\dot{\bm u_f}}_{\ell,\mathcal{K}_\kappa}^2.
\end{equation*}
Then, a similar result hods for the terms $\int_0^t \mathcal{D}(\dot{\UU}-\dot{\UU}_I,\dot{\UU}-\dot{\UU}_I)(s)\, ds$ and $\mathcal{D}(\UU-\UU_I,\UU-\UU_I)(0)$.
Finally, we bound the term $| \UU -  \UU_I |^2_{\Gamma_I}$ by applying the triangle and Cauchy--Schwarz inequalities followed by \cite[Lemma 33]{CangianiDongGeorgoulisHouston_2017}, to infer
\begin{align*}
 | \UU -  \UU_I |^2_{\Gamma_I}  & \lesssim  
 \norm{\alpha^{1/2}(\bm{u}_e - \bm{u}_{eI})}_{\mathcal{F}_h^I}^2  +
\norm{\alpha^{1/2}(\bm{u}_p - \bm{u}_{pI})  }_{\mathcal{F}_h^I}^2 
  + \norm{\gamma^{1/2}(1-\delta)\beta(\bm u_p-\bm u_{pI})\cdot \bm n_p}_{\mathcal{F}_h^I}^2 
 \\ & \quad + \norm{\gamma^{1/2} (\bm u_f-\bm u_{fI})\cdot \bm n_p}_{\mathcal{F}_h^I}^2 
 \\ & \lesssim
  \sum_{\kappa\in\mathcal{T}_h^e}
{\frac{h_\kappa^{2(s_\kappa-1)}}{p_{e,\kappa}^{2m-3}}}\norm{\widetilde{\mathcal{E}}\bm u_e}_{m,\mathcal{K}_\kappa}^2 
+\sum_{\kappa\in\mathcal{T}_h^p}
{\frac{h_\kappa^{2(q_\kappa-1)}}{p_{p,\kappa}^{2n-3}}}\norm{\widetilde{\mathcal{E}}\bm u_p}_{n,\mathcal{K}_\kappa}^2   + \sum_{\kappa\in\mathcal{T}_h^p}
{\frac{h_\kappa^{2(r_\kappa-1)}}{p_{p,\kappa}^{2\ell-3}}}\norm{\widetilde{\mathcal{E}}\bm u_f}_{\ell,\mathcal{K}_\kappa}^2.
\end{align*}
\end{proof}

Before presenting the main result of this Section, we 
set for any time $t\in (0,T]$
\begin{equation*}
    \EE(t) = (\UU-\UU_h)(t) = (\bm u_e-\bm u_{eh},\bm u_p-\bm u_{ph},\bm w_p-\bm w_{ph})(t), 
\end{equation*}
 and use the \textit{strong consistency} of the semi-discrete formulation \eqref{eq::dgsystem_0} to write the following \textit{error equation} 
\begin{equation}\label{eq::error_eq} 
\mathcal{M}(\ddot{\EE},\VV_h) +
\mathcal{D} (\dot{\EE},\VV_h) + \mathcal{A}_h({\EE},\VV_h) +
\mathcal{C}_{h}({\EE},\VV_h) = 0 \quad \forall \, \VV_h \in  \bm{V}_h.
\end{equation}

\begin{theorem}[Semi-discrete error estimates] \label{thm::error-estimate}
Let \eqref{ass::regular},  \eqref{ass::3}, \eqref{ass::2}, and the hypothesis of \eqref{thm:stability} hold. Let the solution $\UU=(\bm u_e, \bm u_p,\bm u_f)$ of problem \eqref{eq::weakform} be such that 
$$
\UU\in C^2([0,T];\bm H^m(\mathcal{T}_h^e)\times \bm H^n(\mathcal{T}_h^p) \times \bm H^\ell(\mathcal{T}_h^p)) \,\cap\, C^1([0,T]; \bm V),
$$ 
with $m,n,\ell\geq 2$ and let $\UU_h = (\bm u_{e},\bm u_{p}, \bm u_{f})_h\in C^2([0,T];\bm{V}_h )$ be the solution of  \eqref{eq::dgsystem_0} with $c_1$ and  $c_2$ sufficiently large. 
Then, the discretization error $\EE=\UU-\UU_h$ satisfies
\begin{multline}\label{eq:error_estimate}
\max_{0\leq t \leq T} \norm{\EE(t)}^2_{\rm E}\lesssim  
\sum_{\kappa\in\mathcal{T}_h^e}{\frac{h_\kappa^{2(s_\kappa-1)}}{p_{e,\kappa}^{2m-3}}}
\mathcal{N}(\bm u_e)^2_{m,\mathcal{K}_\kappa} + 
\sum_{\kappa\in\mathcal{T}_h^p}{\frac{h_\kappa^{2(q_\kappa-1)}}{p_{p,\kappa}^{2n-2}}}
 \mathcal{N}(\bm u_p)^2_{n,\mathcal{K}_\kappa} +  \sum_{\kappa\in\mathcal{T}_h^p}{\frac{h_\kappa^{2(r_\kappa-1)}}{p_{p,\kappa}^{2\ell-3}}}
\mathcal{N}(\bm u_f)^2_{\ell,\mathcal{K}_\kappa},
\end{multline}
with $\mathcal{N}(\bm u)^2_{m,\mathcal{K}_\kappa} = \max_{0\leq t \leq T} \left( \norm{\widetilde{\mathcal{E}}\ddot{\bm u}(t)}_{m,\mathcal{K}_\kappa}^2\hspace{-1mm}+
\norm{\widetilde{\mathcal{E}}\dot{\bm u}(t)}_{m,\mathcal{K}_\kappa}^2\hspace{-1mm}+
\norm{\widetilde{\mathcal{E}}{\bm u}(t)}^2_{m,\mathcal{K}_\kappa} \right)$
and hidden constant depending on the final time $T$ and the material properties, but independent of the discretization parameters.
\end{theorem}
\begin{proof}.
For any time $t\in (0,T]$, let  $\UU_I(t) = (\bm u_e,\bm u_p,\bm u_f)_I(t) \in\bm{V}_h $  be the interpolants defined in \eqref{lemma:interp_estimate}. We split the error as $\EE(t)={\EE}_I(t)-{\EE}_h(t)$, where 
\begin{align*}
{\EE}_I(t)&=(\UU-\UU_I)(t)=(\bm u_e-\bm u_{eI},\bm u_p-\bm u_{pI},\bm u_f-\bm u_{fI})(t),\\
{\EE}_h(t)&=(\UU_h-\UU_I)(t)=(\bm u_{eh}-\bm u_{eI},\bm u_{ph}-\bm u_{pI},\bm u_{fh}-\bm u_{fI})(t).
\end{align*}
From the triangle inequality,  we have
$
\norm{{\EE}(t)}_{ \rm{E}}^2\leq
\norm{{\EE}_h(t)}_{\rm{E}}^2+
\norm{{\EE}_I(t)}_{\rm{E}}^2,
$
and \eqref{lemma:interp_estimate} can be used to bound the term $\norm{{\EE}_I(t)}_{\rm{E}}$.  As for the term $\norm{{\EE}_h(t)}_{\rm{E}}$,
by taking $\VV_h = \dot{\EE}_h \in \bm{V}_h $ as test functions in \eqref{eq::error_eq} and collecting a first time derivative, we obtain
\begin{multline}\label{eq::errors_main}
\frac{1}{2}\frac{d}{dt}\left(\mathcal{M}(\dot{\EE}_h,\dot{\EE}_h) +
\mathcal{A}_h({\EE}_h,{\EE}_h)  + 
\mathcal{C}_{h}({\EE}_h,{\EE}_h)\right) + \mathcal{D} (\dot{\EE}_h,\dot{\EE}_h)= \mathcal{M}(\ddot{\EE}_I,\dot{\EE}_h) \\ 
+\mathcal{D} (\dot{\EE}_I,\dot{\EE}_h) 
-\mathcal{A}_h(\dot{{\EE}}_I,{\EE}_h)
- \mathcal{C}_{h}(\dot{\EE}_I,{\EE}_h)
+ \frac{d}{dt}\mathcal{A}_h({\EE}_I,{\EE}_h)
+\frac{d}{dt} \mathcal{C}_{h}({\EE}_I,{\EE}_h),
\end{multline}
where we have used Leibniz's rule on the term $\mathcal{A}_h(\EE_I,\dot{\EE}_h)$ and $\mathcal{C}_{h}(\EE_I,\dot{\EE}_h)$. 
Now, reasoning as in the proof of Theorem~\ref{teo:stability}, integrating  \eqref{eq::errors_main} between $0$ and $t\le T$, and assuming for simplicity that we can set the initial conditions of the semi-discrete problem so that 
$\dot{\EE}_h(0) = {\EE}_h(0) =\bm{0}$, we get
\begin{multline*}
\norm{{\EE}_h(t)}_{\rm E}^2 + \int_0^t
\mathcal{D} (\dot{\EE}_h,\dot{\EE}_h)(s)\,ds \lesssim  \mathcal{A}_h({\EE}_I,{\EE}_h)(t)
+ \mathcal{C}_{h}({\EE}_I,{\EE}_h)(t) \\
\qquad \qquad + \int_{0}^t\left(\mathcal{M}(\ddot{\EE}_I,\dot{\EE}_h) 
-\mathcal{A}_h(\dot{{\EE}}_I,{\EE}_h)
- \mathcal{C}_{\Gamma_I}(\dot{\EE}_I,{\EE}_h)
 +
\mathcal{D} (\dot{\EE}_I,\dot{\EE}_h)\right)(s)\,ds. 
\end{multline*}
Note that also $\mathcal{D}(\EE_h,\EE_h)(0) = 0$ under the assumption that $\UU_h(0)=\UU_I(0)$.
Then, we apply Cauchy-Schwarz and Young's inequality together with \eqref{eq::continuity_extended_A} to get 
\begin{align*}
\norm{{\EE}_h(t)}_{\rm E}^2  & \lesssim  \trinorm{\EE_I}^2_{\rm dG} + |\EE_I|^2_{\Gamma_I} + \int_0^t \mathcal{D}(\dot{\EE}_I,\dot{\EE}_I)(s)\,ds  \\ 
& + \int_{0}^t (\norm{\ddot{\EE}_I}_{0} +
\trinorm{\dot{\EE}_I}_{\rm dG} + |\dot{\EE}_I|_{\Gamma_I}) (\norm{\dot{\EE}_h}_{0} +
\|\EE_h\|_{\rm dG} + |\EE_h|_{\Gamma_I})(s)\,ds.
\end{align*}
We note that the first three addends can be bounded by $\trinorm{\EE_I(t)}_{\rm E}^2$. Finally, we take the maximum over $t\in[0,T]$, employ the Cauchy-Schwarz and Young's inequality, and observe that $\int_{0}^T |f(s)| ds \leq T \max_{t\in[0,T]} |f(t)|$ to infer 
\begin{align*}
\max_{0\leq t \leq T} \norm{{\EE}_h(t)}_{\rm E}^2  & \lesssim  \max_{0\leq t \leq T} \trinorm{\EE_I(t)}_{\rm E}^2 + T^2 
\max_{0\leq t \leq T} (\norm{\ddot{\EE}_I}_{0}^2 +
\trinorm{\EE_I}_{\rm dG}^2 + |\EE_I|^2_{\Gamma_I})(t) \\
& \lesssim \max_{0\leq t \leq T} \trinorm{\EE_I(t)}_{\rm E}^2 + T^2 \max_{0\leq t \leq T} \trinorm{\dot{\EE}_I(t)}_{\rm E}^2.
\end{align*}
We conclude by applying the estimates in~\eqref{lemma:interp_estimate}.
\end{proof}

\section{dG time discretization}\label{sec:time_integration}
It is known, in the literature, that problems of wave propagation in porous media present difficulties from the point of view of integration over time. This is because in the low-frequency range, as the one we are interested in, the evolution problem becomes stiff \cite{chiavassa_lombard_2013,delapuente2008}, and therefore implicit time integration schemes might be preferred to avoid stability constraint. 
In that respect, to integrate in time the second-order differential system in \eqref{eq::algebraic_compact} we adopt the scheme proposed 
in \cite{AntoniettiMiglioriniMazzieri2021} for the elastodynamics equations, which consists in applying a time dG method to the first-order system
\begin{equation}\label{eq::algebraic_1st}
\left[ \begin{matrix} 
I & 0 \\
0 & M \\
\end{matrix} \right]  
\left[ \begin{matrix} 
\dot{\bm \UU}_h \\
\dot{\bm \VV}_h
\end{matrix} \right] 
+
\left[ \begin{matrix} 
0 & -I \\
A+C & D \\
\end{matrix} \right]  
\left[ \begin{matrix} 
\bm \UU_h \\
\bm \VV_h 
\end{matrix} \right]  = 
\left[ \begin{matrix} 
\bm 0 \\
\bm {\FF}_h 
\end{matrix} \right].
\end{equation}
The latter can be rewritten in a compact form as 
\begin{equation}\label{eq::algebraic_compact_1st}
\widehat{M} \dot{\bm \ZZ}_h 
+ \widehat{K}  \bm \ZZ_h  = \widehat{{\bm \FF}}_h,
\end{equation}
by setting $\bm \ZZ_h = [\bm \UU_h, \bm \VV_h]^T$, and $\widehat{M}$ and $\widehat{K}$ defined as
\begin{equation*}
    \widehat{M} = \left[ \begin{matrix} 
I & 0 \\
0 & M \\
\end{matrix} \right], \quad {\rm and} \quad \widehat{K} = \left[ \begin{matrix} 
0 & -I \\
A+C & D \\
\end{matrix} \right].
\end{equation*}

 Remark that, with respect to the system introduced in \cite{AntoniettiMiglioriniMazzieri2021}, here the first equation of \eqref{eq::algebraic_1st} has not been multiplied by $A$ (that in this case is not positive definite).   
To discretize in time this system we partition the interval $I=(0,T]$ into $N_T$ time-slabs $I_n = (t_{n-1},t_n]$ having length $\Delta t_n = t_n-t_{n-1}$, for $n=1,\dots,N$ with $t_0 = 0$ and $t_N = T$. Then, we introduce the functional spaces
\begin{equation*}
	\label{Eq:L2Space}
	\bm{V}_{\Delta t_n}^{r_n} = \{ \bm \ZZ: I_n \rightarrow \mathbb{R}^{2\ndof} \text{ s.t. } \bm \ZZ \in[\mathcal{P}^{r_n}(I_n)]^{2\ndof}, r_n \geq 1 \},
\end{equation*}
and 
\begin{equation}
	\label{Eq:DGSpace}
	\bm{V}_{\Delta t}^{r}= \{ \bm \ZZ \in \bm L^2(0,T]  \text{ s.t. } \bm \ZZ|_{I_n} = [\bm \UU, \bm \VV]^T|_{I_n} \in \bm{V}_{\Delta t_n}^{r_n} \; \forall \,n = 1,...,N_T \}.
\end{equation} 
We use the notation $[\cdot]_n$ to denote the jump of $\bm \ZZ \in \bm{V}_{\Delta t}^{r}$ at time instant $t_n$, i.e.,  $[\bm \ZZ]_n = \bm \ZZ(t_n^+) - \bm \ZZ(t_n^-)$,  for  $n\ge 0$. With this notation, the time-dG formulation of \eqref{eq::algebraic_compact_1st} reads as: find $\bm \ZZ_{\rm dG}\in \bm{V}_{\Delta t}^{r}$ such that
\begin{equation}
	\label{Eq:WeakProblem}
	\mathcal{A}_{\rm T}(\bm \ZZ_{\rm dG},\bm \WW) = \mathcal{\bm {G}}(\bm \WW) \qquad \forall \; \bm \WW\in \bm{V}_{\Delta t}^{r},
\end{equation} 
where the bilinear form $\mathcal{A}_{\rm T}: \bm{V}_{\Delta t}^{r}\times \bm{V}_{\Delta t}^{r}\rightarrow\mathbb{R}$ is defined by
\begin{align}
	\label{Eq:BilinearForm}
	\mathcal{A}_{\rm T}(\bm \ZZ,\bm \WW) & = \sum_{n=1}^{N_T} (\widehat{M}\dot{\bm \ZZ},{\bm \WW})_{I_n} + 
	(\widehat{K}\bm \ZZ,{\bm \WW})_{I_n} + \sum_{n=1}^{{N_T}-1} \widehat{M} [\bm \ZZ]_n \cdot \bm \WW(t_n^+)+ \widehat{M} \bm \ZZ(0^+)\cdot\bm \WW(0^+), \nonumber
\end{align}
for all $\bm \ZZ, \bm \WW \in \bm{V}_{\Delta t}^{r}$. The linear functional $\mathcal{\bm G}:\bm L^2(0,T)\rightarrow\mathbb{R}$ is defined as 
\begin{equation}
	\label{Eq:LinearFunctional}
	\mathcal{\bm G}(\bm \WW) = \sum_{n=1}^{N_T} (\widehat{\bm \FF}_h, \bm \WW)_{I_n} +  \widehat{M}  \bm \ZZ_{0}\cdot{\bm \WW}(0^+),\;
\end{equation}
for any  $\bm \WW \in \bm{V}_{\Delta t}^{r}$, being $\bm \ZZ_{0} = [\bm \UU_{0h}, \bm \VV_{0h}]^T$.
In practice, to compute the discrete solution to \eqref{Eq:WeakProblem} we iterate over the time intervals, using as initial conditions for $I_{n+1}$ the trace in $t_n$ of the computed solution in the previous interval $I_n$. Hence, problem \eqref{Eq:WeakProblem} written for the generic time interval $I_n$ reduces to: find $\bm \ZZ_{dG}^n \in \bm V_{\Delta t_n}^{r_n}$ such that
\begin{multline}
	\label{Eq:WeakFormulationReduced}
	(\widehat{M} \dot{\bm \ZZ}_{\rm dG}^n,{\bm \WW})_{I_n} +  (\widehat{K} \bm \ZZ_{\rm dG}^n,{\bm \WW})_{I_n} + \widehat{M}  {\bm \ZZ}_{\rm dG}^n(t_{n-1}^+)\cdot {\bm \WW}(t_{n-1}^+)  = (\widehat{\bm \FF}_h,\bm \WW)_{I_n}  + \widehat{M} {\bm \ZZ}_{\rm dG}^{n-1}(t_{n-1}^-) \cdot {\bm \WW}(t_{n-1}^+).
\end{multline}
Focusing on the generic interval $I_n$, we introduce a basis $\{\psi^\ell(t)\}^{\ell =1,...,r_n +1}$ for the polynomial space  $\mathcal{P}^{r_n}(I_n)$ and set ${\dd} =  \ndof (r_n + 1)$ the dimension of the local finite  dimensional space $\bm V_{\Delta t_n}^{r_n}$. We also introduce the vectorial basis 
$\{ {\bm \psi}^{\ell}_i(t) \}^{\ell =1,...,r_n +1}_{i=1,...,2\ndof}$, where $\bm \psi_i^\ell$ is the $2\ndof$-dimensional vector whose $i$-th component is $\psi^\ell$ and the other components are zero.
Using the above notation, we can write the trial function \bm ${\ZZ}_{\rm dG}^n$ as a linear combination of the
basis functions, i.e, 
\begin{equation*}
    \bm \ZZ_{\rm dG}^n(t) = \sum_{j=1}^{2\ndof} \sum_{m=1}^{r_n+1} \alpha_j^m \bm \psi_{j}^m(t), 
\end{equation*}
where $\alpha_j^m \in \mathbb{R}$ for $j = 1, . . . , 2\ndof$ and $m = 1, . . . , r_n + 1$.
Next, we write equation \eqref{Eq:WeakFormulationReduced} for any test function $\bm \psi_{i}^\ell(t)$, $i = 1, . . . , 2\ndof$, $\ell= 1, . . . , r_n + 1$, obtaining the algebraic system of equations
\begin{equation*}
    \bm A^n \bm \alpha^n = \bm b^n, 
\end{equation*}
where on the interval $I_n$, $\bm \alpha^n \in \mathbb{R}^{2\dd}$ is the solution vector, $\bm b^n \in \mathbb{R}^{2\dd}$ corresponds to the data and is given componentwise as
\begin{align*}
   \bm b^n_k & =  (\widehat{\bm \FF}_h,\bm \psi_{i}^\ell)_{I_n} + \widehat{M} {\bm \ZZ}_{\rm dG}^{n-1}(t_{n-1}^-) {\bm \psi_{i}^\ell}(t_{n-1}^+), \quad n = 1,...,N_T, 
\end{align*}
for $k=\ell + (i-1)(r_n+1)$, and  $\ell= 1, . . . , r_n + 1$,  $i = 1, . . . , 2\ndof$. Note that we implicitly use the convention ${\bm \ZZ}_{\rm dG}^{n-1}(0^-) = \bm \ZZ_{0h}$.
The system matrix $\bm A^n \in \mathbb{R}^{2\dd \times 2\dd}$ has the following structure 
\begin{align*}
    \bm A^n & = (N_1 + N_3) \otimes \left[ \begin{matrix}
        I & 0 \\ 0 & M
    \end{matrix}
 \right]  + N_2 \otimes \left[ \begin{matrix} 
0 & -I \\
A+C & D \\
\end{matrix} \right],  \quad n = 1,...,N_T 
\end{align*}
where the time matrices $N_i$, $i=1,2,3$, are defined as follows 
\begin{align*}
    N_1^{\ell m} & = (\dot{\psi}^m,{\psi^\ell})_{I_n}, \quad N_2^{\ell m}  = ({\psi}^m,{\psi^\ell})_{I_n},  \quad N_3^{\ell m} = \psi^m(t_{n-1}^+)\psi^\ell(t_{n-1}^+), 
\end{align*}
for $m,\ell=1,..., r_n+1$. By defining the solution vector and the right-hand side $\bm \alpha^n = [\bm \alpha^n_u,\bm\alpha^n_v ]^T$ and $\bm b^n = [\bm b^n_u, \bm b^n_v]^T$, respectively, and using the properties of the Kronecker product we obtain
\begin{equation}\label{eq:system_xt}
\left(I_{n} \otimes \left[ \begin{matrix}
        I & 0 \\ 0 & M
    \end{matrix}
 \right]  + N_5 \otimes \left[ \begin{matrix} 
0 & -I \\
(A+C) &  D \\
\end{matrix} \right]\right)\left[ \begin{matrix} 
\bm \alpha^n_u \\
\bm \alpha^n_w\\
\end{matrix} \right] = N_4 \otimes 
\left[ \begin{matrix}
        I & 0 \\ 0 & I
    \end{matrix}
 \right]
\left[ \begin{matrix} 
\bm b^n_u \\
\bm b^n_w\\
\end{matrix} \right],   
\end{equation}
where $N_4 = (N_1+N_3)^{-1}$, $N_5=N_4N_2$, and $I_n$ is the identity matrix of order $r_n+1$. 
As noted in \cite[Remark 1]{Southworth2021} the above scheme is equivalent to an Implicit Runge-Kutta (IRK) method with $(r_n + 1)$-stages. 
As in this case, when the entries of the time matrices $N_i$, $i=1,..,3$, and of the right-hand side 
are computed through a Gauss-Legendre-Lobatto (GLL) quadrature formula having $r_n+1$ points and weights, and the basis functions $\psi^\ell$ are the characteristic polynomials associated with that points, one obtains the so-called IRK-Lobatto IIIC schemes (see, e.g. \cite[Section~11.8.3]{qss2007}). As a consequence, scheme \eqref{eq:system_xt} is $L$-stable, algebraically stable, and thus $B$-stable. Moreover, as observed in \cite{Jay2015}, it is perfectly suited for stiff problems.  
According to \cite{butcher2016} Lobatto IIIC methods with $r_n+1$ stages have a maximal order of convergence $2r_n-2$ when a scalar problem is taken into account. However, when the method is applied to a system of ordinary differential equations its order of accuracy can deteriorate 
as observed in \cite{qss2007,Southworth2021}.
In the case under consideration it holds 
\begin{equation*}
    \| \bm \ZZ_h(T^-) - \bm \ZZ_{\rm dG}(T^-)\|_2 \lesssim \Delta t^{r_n+1}, 
\end{equation*}
if the time slabs $I_n$ have all the same length $\Delta t$ for any $n=1,...,N$, cf. Section~\ref{sec:numerical_results}. The formal proof of this result is out of the scope of the paper and will be subject to future research.  

To numerically solve system \eqref{eq:system_xt} 
we apply a block Gaussian elimination getting
\begin{equation}\label{eq:system_xt_gauss}
 \left[ \begin{matrix} I_n \otimes
        I & -N_5 \otimes I\\ 0 & M_w 
    \end{matrix}
 \right] \left[ \begin{matrix} 
\bm \alpha^n_u \\
\bm \alpha^n_w\\
\end{matrix} \right] = 
\left[ \begin{matrix} 
(N_4
\otimes I) \bm b^n_u \\
\bm b^n_w - (N_6\otimes I) \bm b^n_u  \\
\end{matrix} \right],  
\end{equation}
with $M_w = (N_1 + N_3) \otimes M + N_2 \otimes D + N_7 \otimes (A+C)$
being $N_6 = N_2N_4$ and $N_7 = N_2N_4N_2$. Next, we compute $\bm \alpha^n_w$, by solving the linear system $M_w \bm \alpha^n_w = \bm b^n_w - (N_6\otimes I) \bm b^n_u$ and then we update $\bm \alpha^n_u$ by using the first equation.

\section{Numerical results}\label{sec:numerical_results}
In this section, we present numerical results concerning the verification of our scheme on problems with manufactured solutions and the application of the method to cases of geophysical interest. 
In all numerical tests, the penalty parameters $c_1$ and $c_2$ appearing in definitions \eqref{eq::stab_1}--\eqref{eq::stab_2}, respectively, have been chosen equal to 10.

\subsection{Verification test}\label{sec:verification_test}
We consider problem \eqref{eq::elasticity}--\eqref{eq::poroel} coupled with \eqref{eq::contstress_elporo}--\eqref{eq:noflowrate} in $\Omega = \Omega_e \cup \Omega_p$ with $\Omega_e = (0,1)\times(0,1)$ and $\Omega_p = (-1,0) \times (0,1)$. 
We choose the exact solution
\begin{equation}\label{eq:exact_sol_test_case1}
\bm u_e = \cos(4\pi t)\left[ \begin{matrix}
    x^2\sin(2\pi x) \\
    x^2\sin(4\pi x)
\end{matrix}
\right], \quad 
\bm u_p = \cos(\pi \sqrt{2}t )\left[ \begin{matrix}
    x^2\cos(\pi x/2)\sin(\pi x) \\
    x^2\cos(\pi x/2)\sin(\pi x)
\end{matrix}
\right], 
\end{equation}
and $\bm u_f = -\bm u_p$. Dirichlet boundary conditions, initial conditions, and the forcing term $\bm f_e, \bm f_p$, and $\bm g_p$ are set accordingly. For the interface conditions \eqref{eq::contstress_elporo}--\eqref{eq:noflowrate} we choose $\delta = 1$.
The model problem is solved on a sequence of polygonal meshes as the one shown in Figure~\ref{fig:mesh_poligonale}, with adimensional parameters reported in Table~\ref{tab::table_test_case_poroelastic}.

\begin{minipage}{0.5\textwidth}
\vspace{0.5cm}
\hspace{-0.5cm}\includegraphics[width=\textwidth]{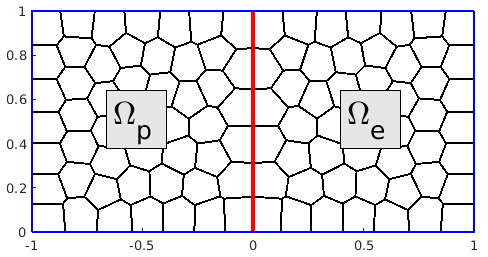}
\captionof{figure}{Test case of Section~ \ref{sec:verification_test}. Example of the computational domain having 100 polygonal elements. The red line represents the interface $\Gamma_I$.}
\label{fig:mesh_poligonale}
\end{minipage}
\hspace{0.05\textwidth}
\begin{minipage}{0.35\textwidth}
\begin{tabular}{llll}
& &  $\Omega_p$ & $\Omega_e$  \\
\hline
\textbf{Fluid}  & $\rho_f$    & 1 & --                 \\ 
& $\eta$      & 1  & --                 \\ \hline
\textbf{Grain}      & $\rho_s$ ($\rho_e$)    & 1 & 1              \\
   & $\mu$       & 1 & 1    \\ 
    & $\zeta$       & 1 & 1    \\ \hline
\textbf{Matrix}       & $\phi$      & 0.5 & --              \\
      & $a$         & 1 & --     \\
     & $k$         & 1 & -- \\
 & $\lambda$ & 1 & 2     \\
 & $m$         & 1 & --    \\
 & $\beta$     & 1 & -- \\  \hline
\end{tabular}
\captionof{table}{Parameters employed for the test case of Section~\ref{sec:verification_test}.}\label{tab::table_test_case_poroelastic}
\end{minipage}

As a first test, we set the final time equal to $1$ and consider, for the time integration scheme in Section~\ref{sec:time_integration}, a timestep $\Delta t_n = \Delta t = 10^{-3}$ and a polynomial degree $r_n=r=1$, for any $n=1,..., 1000$.
In Figure~\ref{fig:conv_test_space} (left),  we report the computed energy error $\|\UU -\UU_h\|_{\rm E}$ at the final time $T$, cf. \eqref{eq::norm_def1} as a function of the mesh size $h$ for a polynomial degree $p_e = p_p = 2, 3, 4$. In this case, we retrieve the rate of convergence $\mathcal{O}(h^p)$ as proved in \eqref{eq:error_estimate}. In Figure~\ref{fig:conv_test_space} (right) we report the same results as before as a function of the polynomial degree $p=p_p=p_e$ obtained by fixing the number of grid elements $N_{el}=100$ and considering $\Delta t_n = \Delta t =5.e-4$ and $r_n=r=1$ for $n=1,...,2000$. 
Notice that this latter case is not covered by our theoretical analysis, nevertheless, we observe numerically optimal convergence. 

On the same numerical example, we compute numerically the $L^2$-norm of the error, i.e., $\| \bm \UU - \bm \UU_h \|_0$, as a function of the time step $\Delta t$, by fixing a polygonal mesh of $N_{el}=100$ elements and the polynomial degree $p=p_p=p_e=7$. We compute the error at the final time $T=1$ by choosing different polynomial degrees $r=1,2,3$. As it can be seen from Figure~\ref{fig:conv_test_dt} the estimated order of convergence is $\mathcal{O}(\Delta t^r)$. Although this is only numerical evidence, it shows that the considered dG method outperforms classical methods such as the Newmark scheme, which is still widely used for wave propagation problems, see e.g., \cite{AntoniettiMazzieriNatipoltri2021}. 
\begin{figure}[ht]
%
%
\begin{tikzpicture}

\begin{axis}[%
width=0.35\textwidth,
height=0.35\textwidth,
scale only axis,
xmode=log,
xmin=0.1,
xmax=0.4,
xminorticks=true,
xlabel style={font=\color{black}},
xlabel={$h$},
ymode=log,
ymin=0.005,
ymax=10,
yminorticks=true,
ylabel style={font=\color{black}},
ylabel={$\| \bm \UU - \bm \UU_h \|_{\rm E}$},
axis background/.style={fill=white},
xmajorgrids,
xminorgrids,
ymajorgrids,
yminorgrids,
legend style={at={(0.63,0.1)}, anchor=south west, legend cell align=left, align=left, draw=black}
]
\addplot [color=blue, line width=2.0pt, mark=asterisk, mark options={solid, blue}]
  table[row sep=crcr]{%
0.324044282389075	3.500122009837051e+00\\
0.249512906100889	2.864065324925036e+00\\
0.178218721248557	1.532951587515947e+00\\
0.12716409863985	7.783025248444815e-01\\
};
\addlegendentry{$p=2$}

\addplot [color=red, line width=2.0pt, mark=star, mark options={solid, red}]
  table[row sep=crcr]{%
0.324044282389075	2.532885618327975e+00\\
0.249512906100889	7.753350377418354e-01\\
0.178218721248557	2.885689865253659e-01\\
0.12716409863985	1.029240773482629e-01\\
};
\addlegendentry{$p=3$}

\addplot [color=green, line width=2.0pt, mark=square, mark options={solid, green}]
  table[row sep=crcr]{%
0.324044282389075	2.440166288155848e-01\\
0.249512906100889	1.108858631372115e-01\\
0.178218721248557	2.961244391469967e-02\\
0.12716409863985	8.764810870650484e-03\\
};
\addlegendentry{$p=4$}

\addplot [color=black, forget plot]
  table[row sep=crcr]{%
0.178218721248557	2.17619126034708\\
0.12716409863985	1.21707079828856\\
};

\addplot [color=black, forget plot]
  table[row sep=crcr]{%
0.178218721248557	2.17619126034708\\
0.12716409863985	2.17619126034708\\
};
\addplot [color=black, forget plot]
  table[row sep=crcr]{%
0.12716409863985	2.17619126034708\\
0.12716409863985	1.21707079828856\\
};

\node[right, align=left, text=black, font=\normalsize]
at (axis cs:0.115,1.6) {$2$};

\addplot [color=black, forget plot]
  table[row sep=crcr]{%
0.178218721248557	0.424542558644923\\
0.12716409863985	0.15422501287589\\
};
\addplot [color=black, forget plot]
  table[row sep=crcr]{%
0.178218721248557	0.424542558644923\\
0.12716409863985	0.424542558644923\\
};
\addplot [color=black, forget plot]
  table[row sep=crcr]{%
0.12716409863985	0.424542558644923\\
0.12716409863985	0.15422501287589\\
};

\node[right, align=left, text=black, font=\normalsize]
at (axis cs:0.115,0.25) {$3$};

\addplot [color=black, forget plot]
  table[row sep=crcr]{%
0.178218721248557	0.0504409546115258\\
0.12716409863985	0.0130745898333879\\
};
\addplot [color=black, forget plot]
  table[row sep=crcr]{%
0.178218721248557	0.0504409546115258\\
0.12716409863985	0.0504409546115258\\
};
\addplot [color=black, forget plot]
  table[row sep=crcr]{%
0.12716409863985	0.0504409546115258\\
0.12716409863985	0.0130745898333879\\
};

\node[right, align=left, text=black, font=\normalsize] 
at (axis cs:0.115,0.025) {$4$};

\end{axis}

\end{tikzpicture}%
%
%
\begin{tikzpicture}

\begin{axis}[%
width=0.35\textwidth,
height=0.35\textwidth,
scale only axis,
scale only axis,
xmin=1,
xmax=5,
xlabel style={font=\color{black}},
xlabel={$p=p_p=p_e$},
ymode=log,
ymin=0.005,
ymax=10,
yminorticks=true,
ylabel style={font=\color{black}},
ylabel={$\| \bm \UU - \bm \UU_h \|_{\rm E}$},
axis background/.style={fill=white},
xmajorgrids,
ymajorgrids,
yminorgrids
]
\addplot [color=blue, line width=2.0pt, mark=asterisk, mark options={solid, blue}, forget plot]
  table[row sep=crcr]{%
1	9.450073931402679e+00\\
2	2.988659507661924e+00\\
3	7.863233811892762e-01\\
4	1.123760677302777e-01\\
5	1.551172958514478e-02\\
};
\end{axis}

\end{tikzpicture}%
\caption{Test case of Section~\ref{sec:verification_test}. 
Left: computed energy-errors $\| \bm \UU - \bm \UU_h\|_{\rm E}$ at $T=1$ in logarithmic scale as a function of the mesh size $h$ for different polynomial degrees $p=p_e=p_p$ and  fixing the time step $\Delta t=0.001$ and  the time polynomial degree $r=1$. The rate of convergence is in agreement with the theoretical estimates in \eqref{eq:error_estimate}.
Right: computed energy-errors $\| \bm \UU - \bm \UU_h\|_{\rm E}$ at $T=1$ in semilogarithmic scale as a function of the polynomial degree $p=p_p=p_e$ by fixing the number of elements $N_{el}=100$, the time step $\Delta t=5.e-4$, and the time polynomial degree $r=1$.
}\label{fig:conv_test_space}
\end{figure}
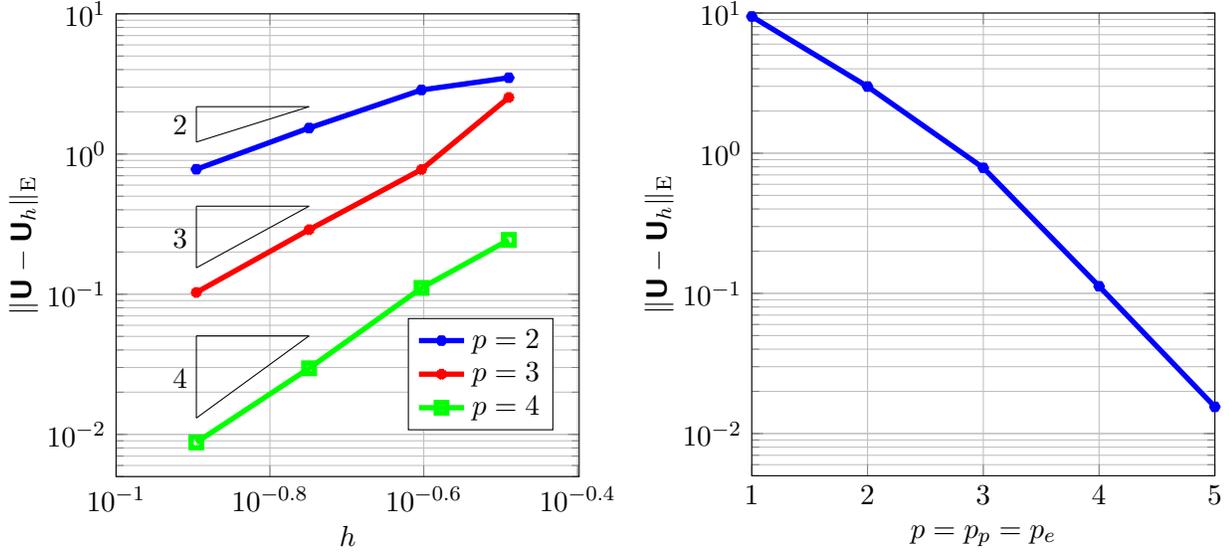
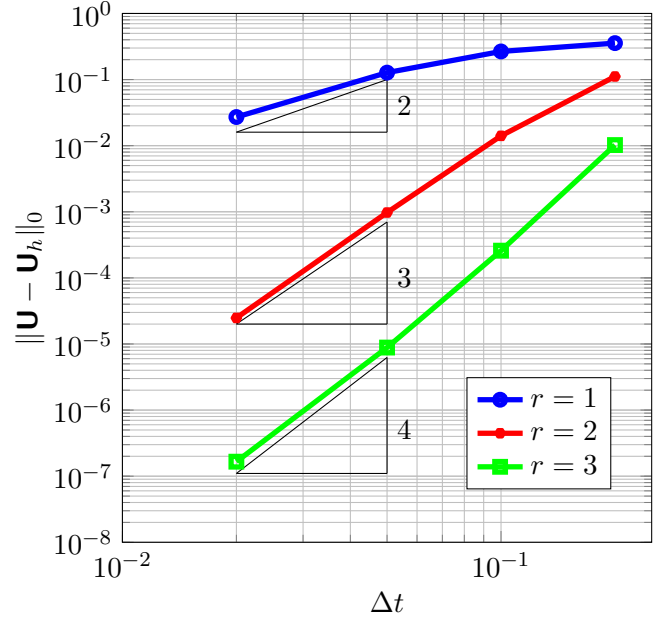
\begin{figure}[h!]
\centering
%
%
\begin{tikzpicture}

\begin{axis}[%
width=0.4\textwidth,
height=0.4\textwidth,
scale only axis,
xmode=log,
xmin=0.01,
xmax=0.25,
xminorticks=true,
xlabel style={font=\color{black}},
xlabel={$\Delta t $},
ymode=log,
ymin=1e-08,
ymax=1,
yminorticks=true,
ylabel style={font=\color{black}},
ylabel={$\| \bm \UU - \bm \UU_h \|_{0}$},
axis background/.style={fill=white},
xmajorgrids,
xminorgrids,
ymajorgrids,
yminorgrids,
legend style={at={(0.65,0.1)}, anchor=south west, legend cell align=left, align=left, draw=black}
]
\addplot [color=blue, line width=2.0pt, mark=o, mark options={solid, blue}]
  table[row sep=crcr]{%
0.2	3.559324637021170e-01\\
0.1	2.662384016385438e-01\\
0.05	1.273617217387445e-01\\
0.02	2.711058921788446e-02\\
};
\addlegendentry{$r=1$}

\addplot [color=red, line width=2.0pt, mark=asterisk, mark options={solid, red}]
  table[row sep=crcr]{%
0.2	1.113419655045532e-01\\
0.1	1.408302506655934e-02\\
0.05	9.739505804510285e-04\\
0.02	2.477955516645199e-05\\
};
\addlegendentry{$r=2$}

\addplot [color=green, line width=2.0pt, mark=square, mark options={solid, green}]
  table[row sep=crcr]{%
0.2	1.021601755762847e-02\\
0.1	2.581670683458181e-04\\
0.05	8.828227698246276e-06\\
0.02	1.657662810791810e-07\\
};
\addlegendentry{$r=3$}

\addplot [color=black, forget plot]
  table[row sep=crcr]{%
0.05	0.1\\
0.02	0.016\\
};
\addplot [color=black, forget plot]
  table[row sep=crcr]{%
0.05	0.016\\
0.02	0.016\\
};
\addplot [color=black, forget plot]
  table[row sep=crcr]{%
0.05	0.1\\
0.05	0.016\\
};

\addplot [color=black, forget plot]
  table[row sep=crcr]{%
0.05	0.0007\\
0.02	0.00002\\
};
\addplot [color=black, forget plot]
  table[row sep=crcr]{%
0.05	0.00002\\
0.02	0.00002\\
};
\addplot [color=black, forget plot]
  table[row sep=crcr]{%
0.05	0.0007\\
0.05	0.00002\\
};

\addplot [color=black, forget plot]
  table[row sep=crcr]{%
0.05	6.25e-06\\
0.02	1.1e-07\\
};
\addplot [color=black, forget plot]
  table[row sep=crcr]{%
0.05	1.1e-07\\
0.02	1.1e-07\\
};
\addplot [color=black, forget plot]
  table[row sep=crcr]{%
0.05	6.25e-06\\
0.05	1.1e-07\\
};

\node[right, align=left, text=black, font=\normalsize]
at (axis cs:0.05,0.04) {$2$};

\node[right, align=left, text=black, font=\normalsize]
at (axis cs:0.05,0.0001) {$3$};

\node[right, align=left, text=black, font=\normalsize]
at (axis cs:0.05,0.0000005) {$4$};

\end{axis}

\end{tikzpicture}%
\caption{Test case of Section~\ref{sec:verification_test}. 
Computed $L^2$-errors $\| \bm \UU - \bm \UU_h\|_{0}$ at $T=1$ in logarithmic scale as a function of the time step $\Delta t$ for different polynomial degrees $r=1,2,3$ in time. We set $N_{el}=100$ polygonal elements and a space polynomial degree $p=p_p=p_e=7$. 
}\label{fig:conv_test_dt}
\end{figure}
Next, in Table~\ref{tab::table_hdt_conv_p1} we report the computed $L^2$-error $\| \bm \UU - \bm \UU_h\|_{0}$ as a function of the discretization parameters. In particular, we fix the polynomial degree for both space and time variables and we let $N_{el}$ and $\Delta t$ vary.   
It is possible to notice that the spatial discretization error is dominant since we obtain an almost constant value for each row of  Table~\ref{tab::table_hdt_conv_p1}. 
It is interesting to analyze these results in connection with the condition number of the system matrix $M_w$, cf. Figure~\ref{fig:cond_number} (first row)  and the computational time 
spent for the single run, cf. Figure~\ref{fig:cond_number} (second row). First of all, we can observe that the condition number of the system matrix increases by one order whenever the polynomial degree increases by one. Moreover, when fixing the polynomial degree, the matrix $M_w$ is better conditioned for a smaller value of $\Delta t$, see Figure~\ref{fig:cond_number} (first row).  
Concerning the computational cost, it is obvious that this is proportional to the dimension of the system matrix. Looking at the plot in Figure~\ref{fig:cond_number} (second row) one can observe that even if different combinations of discretization parameters can lead to the same amount of time spent for a single simulation, they do not provide the same level of accuracy. Indeed, the sets $(p=r=3, \Delta t=0.01$, $N_{el}=400)$ and $(p=r=4, \Delta t=0.02$, $N_{el}=200)$ are equivalent from the point of view of the computational cost ($1389~s$ and $1130~s$, respectively) but, with the first set we obtain an $L^2-$error equal to $1.7814e$-04, while with the second we get $5.6899e$-05.

\begin{table}
\centering
\begin{tabular}{|c|c|c|c|c|c|}
\cline{1-6}
& $N_{el}$ $\backslash$ $\Delta t$ & 0.1 & 0.05 & 0.02 & 0.01 \\
\cline{1-6}
\multirow{4}{*}{$p=r=1$}& 50  & 2.5734e-01  & 2.3169e-01 &  2.1453e-01 &  2.0908e-01 \\
 & 100  & 2.4613e-01  & 2.0702e-01 &  1.9284e-01 &  1.9759e-01 \\
 & 200  & 2.4419e-01  & 1.9706e-01 &  1.8611e-01 & 1.8736e-01 \\
 & 400 & 2.5188e-01  & 2.0336e-01 &  1.7338e-01 &  1.6802e-01 \\
\cline{1-6} 
 \multirow{4}{*}{$p=r=2$}& 50 & 4.8538e-02  &  4.8105e-02  & 4.8164e-02  & 4.8125e-02 \\
 & 100 &  5.1387e-02  &  4.7603e-02 &  4.7353e-02 &  4.7406e-02 \\
 & 200 &  3.2140e-02  & 2.7143e-02 &  2.6818e-02  & 2.6841e-02 \\
&  400 &  1.5432e-02  & 8.3925e-03 &  8.0168e-03  & 8.0096e-03 \\
 \cline{1-6} 
 \multirow{4}{*}{$p=r=3$}& 50 &  1.5983e-02 &  1.5933e-02  & 1.5947e-02  & 1.5943e-02 \\
 & 100 &  3.1176e-03 &  3.0956e-03 &  3.1093e-03 &  3.1157e-03 \\
& 200 &  8.7709e-04  & 8.4075e-04 &  8.4631e-04 &  8.5067e-04 \\
& 400 &  2.7540e-04  & 1.7519e-04 &  1.7682e-04 &  1.7814e-04 \\
\cline{1-6}
\multirow{4}{*}{$p=r=4$}& 50 &   9.1399e-04  & 9.2689e-04 &  9.3522e-04  & 9.3748e-04 \\ & 100 &  2.8139e-04  & 2.8182e-04  & 2.8601e-04 &
 2.8739e-04 \\
  & 200 &   5.6281e-05 &  5.6095e-05 &  5.6899e-05 &  5.7177e-05 \\
  & 400 & 1.3301e-05  & 1.0781e-05 &  1.0890e-05  & 1.1072e-05 \\
\cline{1-6}
\end{tabular}
\caption{Test case of Section~\ref{sec:verification_test}. Computed $L^2$-errors $\| \bm \UU - \bm \UU_h\|_{0}$ as function of the discretization parameters. }\label{tab::table_hdt_conv_p1}
\end{table}

\begin{figure}[h!]
\hspace{-1.2cm}\includegraphics[width=1.\textwidth]{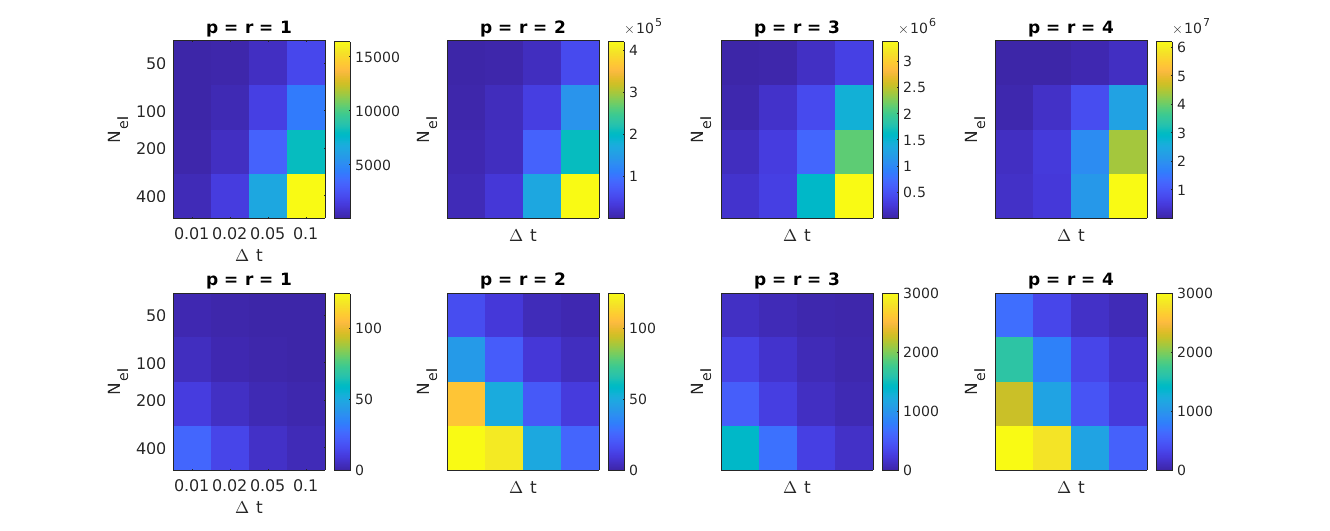}
\caption{Test case of Section~\ref{sec:verification_test}. 
First row: computed condition number (\texttt{condest} function in Matlab) as a function of the discretization parameters. 
Second row: computational time employed for a single run as a function of the discretization parameters. 
}\label{fig:cond_number}
\end{figure}

\subsection{Wave propagation in a two layer medium}\label{sec:layered_media}

Inspired by \cite{Morency2008}, we consider a wave propagation problem in a two-layered medium. 
The computational domain $\Omega = (0, 4800~m)^2$ and consists of two layers as shown in Figure~\ref{fig:domain_layered_media} (left). We assume the upper layer to be  a poroelastic material while the lower layer to be an elastic medium, cf. Table~\ref{tab::table_poroelastic}.

\begin{figure}[h!]
\includegraphics[width=0.45\textwidth]{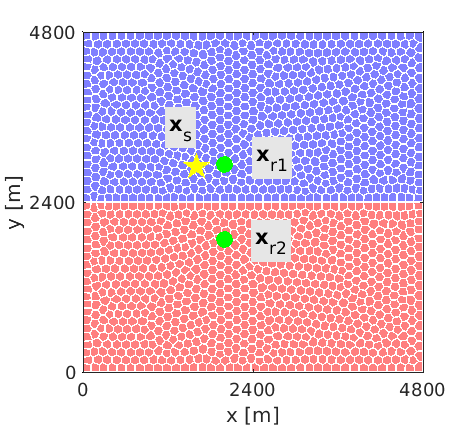}\hspace{5mm}
\includegraphics[width=0.45\textwidth]{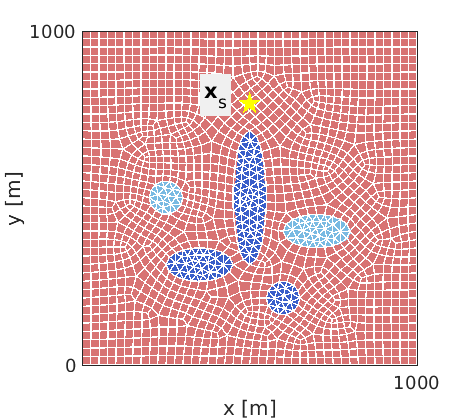} 
\caption{Left: computational domain for the test case of Section~\ref{sec:layered_media}.
Right: computational domain for the test case of Section~\ref{sec:complex_media}. Yellow stars denote the nucleation point $\bm x_s$, while green dots denote the position of the receivers $\bm x_{r1}$ and $\bm x_{r2}$.
}\label{fig:domain_layered_media}
\end{figure}

\begin{table}[htbp]
\begin{tabular}{lllllll}
& & & Proelastic  Layer & & \\
\hline 
\textbf{Fluid}  & Fluid density      & $\rho_f$     & 950                 & $\rm kg/m^3$    &  \\
& Dynamic viscosity  & $\eta$        & 0 (0.0015)                   & $\rm Pa\cdot s$ &  \\ \hline
\textbf{Grain}  & Solid density      & $\rho_s$     & 2200                 & $\rm kg/m^3$    &  \\
& Shear modulus      & $\mu$        & 4.3738 $\cdot 10^9$     & $\rm Pa$        &  \\ \hline
\textbf{Matrix} & Porosity           & $\phi$       & 0.4                 &             &  \\
& Tortuosity         & $a$         & 2                  &             &  \\
& Permeability       & $k$          & $1\cdot 10^{-12}$ & $\rm m^2$       &  \\
& Lam\'e coefficient   & $\lambda$  & 7.2073$\cdot 10^9$    & $\rm Pa$        &  \\
& Biot's coefficient & $m$          & 6.8386$\cdot 10^9$   & $\rm Pa$        &  \\
& Biot's coefficient & $\beta$      & 0.0290                 &             &  \\  
 & Damping coefficient   & $\zeta$ & 0 (0.01)    & $\rm s^{-1}$        &  \\
\hline \\
& & & Elastic  Layer & & \\
\hline
 \textbf{Matrix} & Solid density      & $\rho$    & 2650                 & $\rm kg/m^3$    &  \\
& Shear modulus      & $\mu$       & 1.5038 $\cdot 10^9$    & $\rm Pa$        &  \\
 & Lam\'e coefficient   & $\lambda$ & 1.8121$\cdot 10^9$   & $\rm Pa$        &  \\
 & Damping coefficient   & $\zeta$ & 0 (0.01)    & $\rm s^{-1}$        &  \\
  \hline
\end{tabular}
\caption{Test case of Section~\ref{sec:layered_media}. Physical parameters for the layered media.}
\label{tab::table_poroelastic}
\end{table}

An explosive source is located in the upper layer at $\bm x_s = (1600, 2900)~m$ whose expression is given by
\begin{equation}\label{eq:source_time}
    \bm f_p = \bm g_p = - M \cdot \nabla \delta (\bm x - \bm x_s)S(t),
\end{equation}
being $M=M_0 I$ the moment tensor with $M_0>0$, $\delta(\bm x - \bm x_s)$ is the Dirac delta distribution centered in $\bm x_s$ and $S(t)$ is the source time function. This is a classical choice in the context of earthquake simulation, cf. \cite{Morency2008}.
We consider as a time evolution for $S(t)$ in \eqref{eq:source_time} a Ricker-wavelet 
\begin{equation}\label{eq:Ricker}
    S(t) = (1-2 \beta_p (t-t_0)^2) e^{ -\beta_p (t-t_0)^2}, \quad \beta_p = \pi^2 f_p^2,
\end{equation}
with time-shift $t_0 = 0.3~s$ and peak-frequency $f_p = 5~ Hz$. Finally, we set $M_0 = 1~Nm$.
We use a shape-regular polygonal mesh with characteristic size $h = 100$ and a polynomial degree $p = 3$ for space discretization. For time integration, we set $\Delta t = 0.01$, a polynomial degree $r=2$ and we fix the final time $T=1.5~s$.
For this model, we consider interface  conditions \eqref{eq::contstress_elporo}--\eqref{eq:noflowrate} with $\delta = 1$, free surface boundary conditions on the top boundary, i.e. $\bm \sigma_p \bm n_p = \bm 0$, and absorbing boundary conditions on the remaining part of the boundary to avoid artificial reflections and simulate an unbounded medium.
In particular, we use the classical first-order paraxial approximations proposed in \cite{Stacey} for the elastic boundary $\Gamma_e$, i.e., 
\begin{align}\label{eq:abso_el}
\bm\sigma_e \bm n_e  &= \rho_(c_p - c_s)(\partial_t \bm u_e \cdot \bm n_e)\bm n_e + \rho_e c_s \partial_t \bm u_e,
\end{align}
and the ones proposed in  \cite{Morency2008} for the poroelastic boundary $\Gamma_p$, i.e.,
\begin{align}\label{eq:abso_poro}
\bm\sigma_p \bm n_p & = \rho_p c_{pI} (\partial_t \bm u_p \cdot \bm n_p) \bm n_p + \rho_f c_{pII} (\partial_t \bm u_f \cdot \bm n_p) \bm n_p  + ( \rho_p - \rho_f \phi/a)c_s (I -\bm n_p \bm n_p) \cdot \partial_t \bm u_p,\nonumber \\
-p \bm n_p & =\rho_f a/\phi c_{pII} (\partial_t \bm u_f \cdot \bm n_p) \bm n_p + \rho_f c_{pI} (\partial_t \bm u_p \cdot \bm n_p) \bm n_p. 
\end{align}
In \eqref{eq:abso_el} $c_p = \sqrt{\frac{\lambda +2 \mu}{\rho_e}}$ and $c_s = \sqrt{\frac{\mu}{\rho_e}}$ are the compressional and shear wave velocities, respectively, while in  \eqref{eq:abso_poro} $c_{pI}$ and $c_{pII}$ are the fast and slow compressional wave velocities, respectively, defined as  $ c_{pI} = \max \sqrt{\Lambda}$ and $  c_{pII} = \min \sqrt{\Lambda}$,  where $\Lambda$ are the solutions of the generalized eigenvalue problem $A \bm v = \Lambda B \bm v$ where
\begin{equation*}
    A = \left[ \begin{matrix}
        \rho_p & \rho_f \\ \rho_f & \rho_w
    \end{matrix} \right], \quad B = \left[ \begin{matrix}
        \lambda + 2 \mu + m\beta^2 & m\beta \\ m\beta & m
    \end{matrix} \right].
\end{equation*}

In Figure~\ref{fig:snap_tromp}, we report some snapshots of the vertical component velocity $(\dot{\bm u}_p,\dot{\bm u}_e)_y$ obtained by 
neglecting or considering viscous effects in the model. In particular, we consider the model $TC_1$ where $\eta=\zeta = 0$ (Figure~\ref{fig:snap_tromp}-first row), and the model $TC_2$ where we set $\eta=0.0015$ and  $\zeta=0.01$ (Figure~\ref{fig:snap_tromp}-second row). The results of $TC_1$ are in agreement with the ones presented in \cite{Morency2008,Peng2020}.  Indeed, from Figure~\ref{fig:snap_tromp} (first row), we clearly observe: (i) the continuity of the velocity field across the interface $\Gamma_I=(0,4800)~m\times\{2400\}~m$; (ii) the propagation of the direct fast $c_{pI}$  compressional wave (first front), the reflected fast $c_{pI}$ compressional wave (second front), the reflected shear $c_s$ (third front) and slow compressional $c_{pII}$ wave (fourth front) in the upper poroelastic layer; (iii) the transmitted compressional $c_p$ and shear $c_s$ wave in the lower elastic layer. 
The same physical phenomena, although less evident, are present when the viscous damping terms are introduced in $TC_2$, cf. Figure~\ref{fig:snap_tromp}(second row). In this case, the transmitted shear $c_s$ and the reflected slow longitudinal $c_{PII}$ waves are too weak to be visible at this scale. 
From the plot, we can also see the effect of the first-order absorbing boundaries, which are not perfectly transparent. Recent studies have adopted the more efficient perfectly matched layer (PML) methodology, e.g. \cite{Martin2008,HE2019116}. The PML implementation is not included in this paper but it will be addressed in the next future.

\begin{figure}[h!]
\includegraphics[width=1\textwidth]{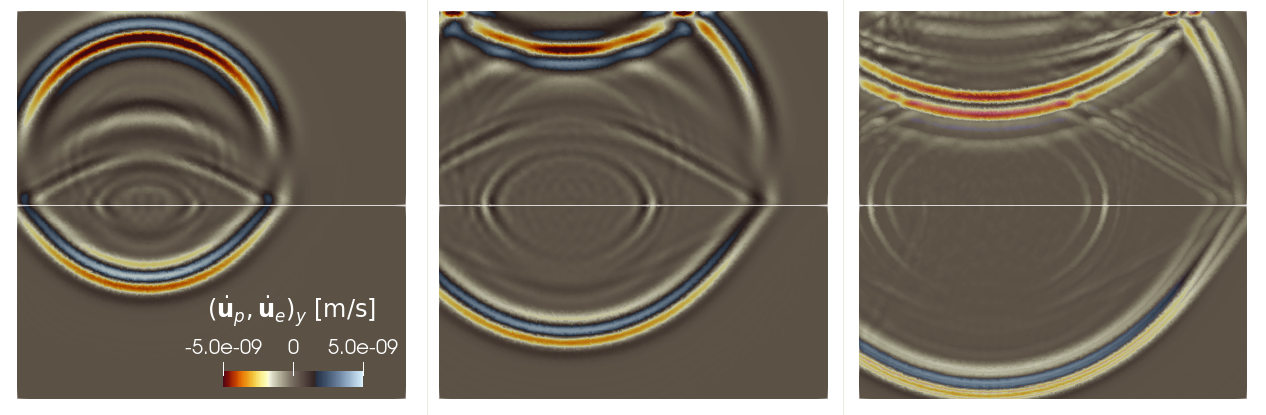}
\includegraphics[width=1\textwidth]{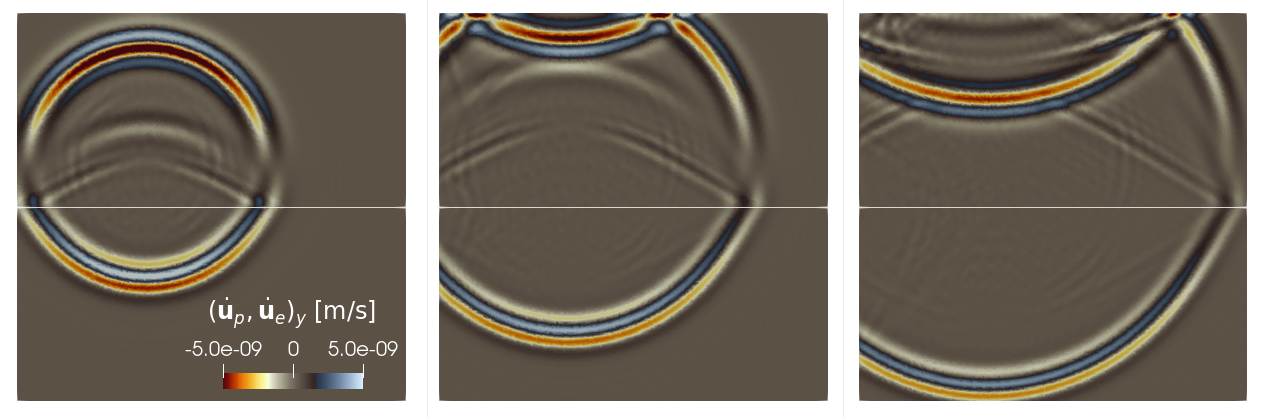} 
\caption{Test case of Section~\ref{sec:layered_media}. 
Snapshots of the vertical component velocity $(\dot{\bm u}_p, \dot{\bm u}_e)_y$ at different time instants: $t=0.9s$ (left), $t=1.2s$ (center), and $t=1.5s$ (right). The top and bottom rows refer to the model $TC_1$ and $TC_2$, respectively. 
}\label{fig:snap_tromp}
\end{figure}

In Figure~\ref{fig:th_Tromp} we compare the 
the time histories velocities $(\dot{\bm u}_p,\dot{\bm u}_e)$ at the two receivers $\bm x_{r1} = (2000, 2934)~m$ and at $\bm x_{r2} = (2000, 1867)~m$ 
for the different modeling assumptions. 
Again here we notice the effect of the damping assumption on the scattered wave field.

\begin{figure}[h!]
\includegraphics[width=1\textwidth]{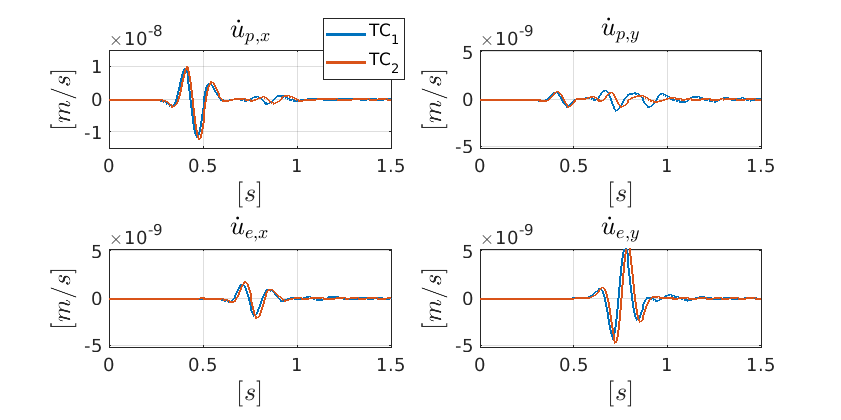}
\caption{Test case of Section~\ref{sec:layered_media}. 
Horizontal and vertical component velocity $\dot{\bm u}_p$ (top) and  $\dot{\bm u}_e$ (bottom) for the receivers $\bm x_{r1}$ (top) and $\bm x_{r2}$ (bottom), respectively.
}\label{fig:th_Tromp}
\end{figure}

\subsection{Wave propagation in an elastic domain with poroelastic inclusions}\label{sec:complex_media}

To demonstrate the feasibility of tackling a more complex model with the proposed method, we consider a coupled elastic-poroelastic model shown in Figure~\ref{fig:domain_layered_media} (right).
The background media is regarded as a perfectly elastic medium (whose mechanical properties are listed in Table~\ref{tab::table_poroelastic}) while the circle and ellipses with different sizes represent oil and gas reservoirs, which are modeled as poroelastic media. This example takes inspiration from the one proposed in \cite{Zhang2019}. The poroelastic domains have different material properties: for the domains depicted in light blue in Figure~\ref{fig:domain_layered_media} (right) we consider the values in Table~\ref{tab::table_poroelastic} while for the remaining ones, we use the ones in Table~\ref{tab::table_poroelastic_2}.

\begin{table}[htbp]
\centering
\begin{tabular}{lllllll}
& & & Proelastic  Layer & & \\
\hline
\textbf{Fluid}  & Fluid density      & $\rho_f$     & 750                 & $\rm kg/m^3$    &  \\
& Dynamic viscosity  & $\eta$        & 0                    & $\rm Pa\cdot s$ &  \\ \hline
\textbf{Grain}  & Solid density      & $\rho_s$     & 2650                 & $\rm kg/m^3$    &  \\
& Shear modulus      & $\mu$        & 1.503 $\cdot 10^9$     & $\rm Pa$        &  \\ \hline
\textbf{Matrix} & Porosity           & $\phi$       & 0.2                &             &  \\
& Tortuosity         & $a$         & 2                  &             &  \\
& Permeability       & $k$          & $1\cdot 10^{-12}$ & $\rm m^2$       &  \\
& Lam\'e coefficient   & $\lambda$  & 1.8121$\cdot 10^9$    & $\rm Pa$        &  \\
& Biot's coefficient & $m$          & 7.2642$\cdot 10^9$   & $\rm Pa$        &  \\
& Biot's coefficient & $\beta$      & 0.9405                 &             &  \\  
 & Damping coefficient   & $\zeta$ & 0     & $\rm s^{-1}$        &  \\
\hline
\end{tabular}
\caption{Test case of Section~\ref{sec:complex_media}: physical parameters for the dark blue poroelastic domain in Figure~\ref{fig:domain_layered_media} (right).}
\label{tab::table_poroelastic_2}
\end{table}

We consider a seismic source as in the previous test case applied to the point $\bm x_s =(500, 780)~m$, with time variation given by a Ricker wavelet with time shift $t_0=0.2~s$ and peak frequency $f_p = 10~Hz$, cf. \eqref{eq:Ricker}.
We consider absorbing boundary conditions for the external elastic domain, cf. \eqref{eq:abso_el}, while we choose different values for $\delta$ 
in \eqref{eq::contstress_elporo}--\eqref{eq:noflowrate}. In particular, we consider $\delta =0,\frac12,1$.
For the numerical discretization, we consider a polygonal decomposition with both quadrilateral and triangular elements of characteristic size $h=25$, cf. Figure~\ref{fig:domain_layered_media} (right), a polynomial degree $p=r=3$, a time step $\Delta t = 0.01~s$ and a final time $T=1~s$. 

In Figure~\ref{fig:snap_complex} we report the snapshots at different times of the computed vertical velocity, for different values of $\delta$. From the results reported in Figure~\ref{fig:snap_complex}, we see that when the seismic waves, generated by the seismic source, reach the oil and gas reservoirs, they are reflected and transmitted. The seismic waves are clearly shown, and we can claim that there is no significant numerical dispersion. We can see the absorbing boundary conditions perform rather well for this complex model, and the seismic waves are absorbed by the boundary elements.
The effect of $\delta$ is also visible, the more $\delta$ is close to zero and the more the waves remain trapped inside the gas reservoirs. This is in agreement with the model hypothesis in Section~\ref{sec::physical} (see also \cite{AntoniettiMazzieriNatipoltri2021}).

\begin{figure}[h!]
\includegraphics[width=1\textwidth]{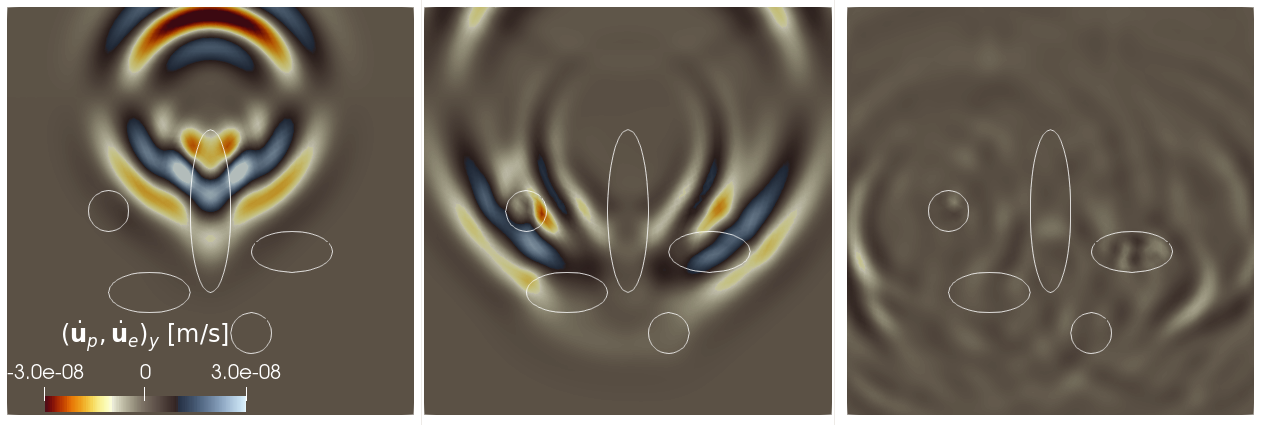}
\includegraphics[width=1\textwidth]{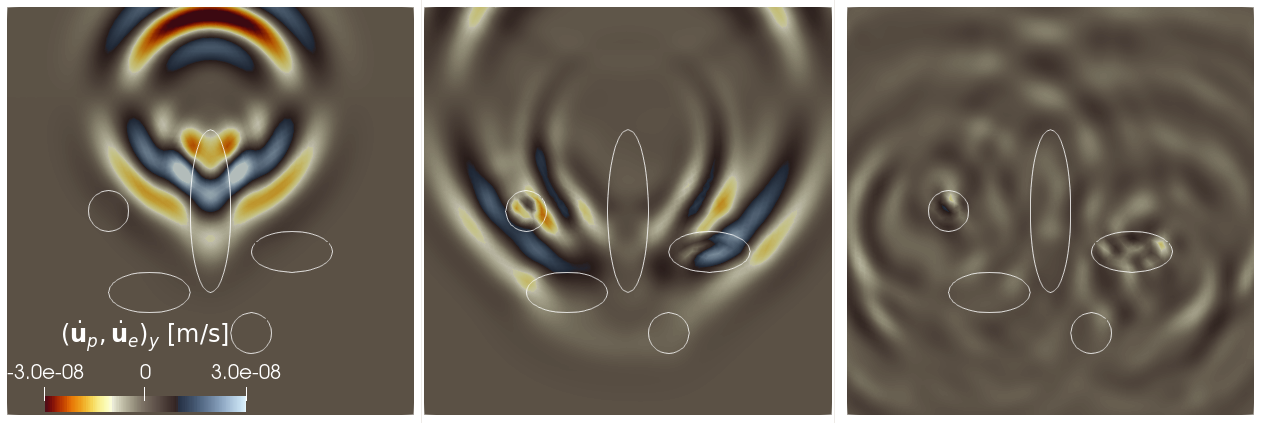}
\includegraphics[width=1\textwidth]{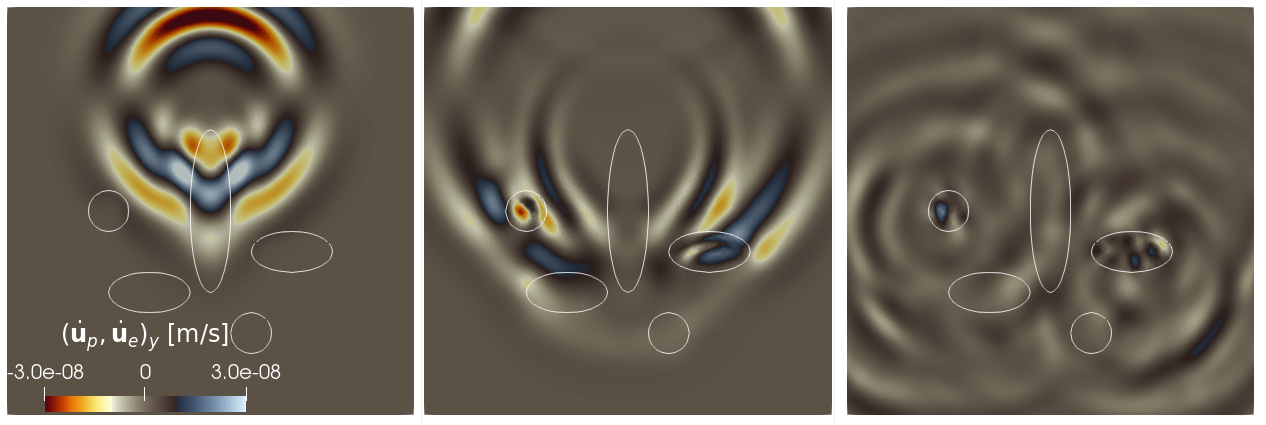}
\caption{Test case of Section~\ref{sec:complex_media}. 
Snapshots of the vertical component velocity $(\dot{\bm u}_p, \dot{\bm u}_e)_y$ at different time instants: from left to right $t=0.3,0.4,0.7~s$; top row $\delta=1$, middle row $\delta =0.5$, bottom row $\delta=0$.}
\label{fig:snap_complex}
\end{figure}

\section{Conclusions}\label{sec:conclusions}

In this work, we have presented a space-time 
PolydG methods for wave propagation problems in coupled poroelastic-elastic media.
Based on a displacement weak formulation of the problem, we proved stability and error bound for the semi-discretization, where (an interior penalty type) PolydGdG method is considered. Time integration is achieved by an unconditionally stable and implicit dG scheme, which also guarantees high-order accuracy in the time domain. Numerical experiments have been designed not only to verify the numerical performance of the space-time PolydG method but also to exploit the flexibility in the process of mesh design offered by polytopic elements. In this respect, numerical tests of geophysical interest have been also discussed. The presented space-time PolydG method allows a robust and flexible numerical discretization that can be successfully applied to multiphysics wave propagation problems. Future developments in this direction include the extension to realistic three-dimensional problems, the coupling of this model to fluid-structure (with poroelastic, thermo-elastic, or acoustic structure) interaction problems as well as the design of efficient solvers for the solution of the (linear) system of equations stemming from this space-time PolydG discretization.

\section{Acknowledgments} 
P.F.A. has been partially funded by the research grants PRIN2017 n. 201744KLJL and PRIN2020 n. 20204LN5N5 funded by the Italian Ministry of Universities and Research (MUR).  P.F.A. and I.M. have been partially funded by ICSC—Centro Nazionale di Ricerca in High Performance Computing, Big Data, and Quantum Computing funded by European Union—NextGenerationEU.
M.B., I.M., and P.F.A. are members of INdAM-GNCS. 
The work of M.B. has been partially supported by the INdAM-GNCS project CUP E55F22000270001.
The work of I.M. has been partially supported by the INdAM-GNCS project CUP E53C22001930001.
The present research has been partially supported by MUR, grant Dipartimento di Eccellenza 2023-2027

\appendix
\section{Definition of the coupling bilinear form}
The aim of this appendix is to motivate the definition of $\mathcal{C}_{h}(\cdot,\cdot)$ adopted in Section~\ref{sec:PolydG_form}.
We consider equations \eqref{eq::elasticity} and \eqref{eq::poroel}, and we multiply them by test functions $\bm w_e \in \bm V_h^e$ and $\bm w_p, \bm w_f \in \bm V_h^p$, respectively, integrate by parts element-wise and focus only to the interface terms, i.e., integrals on $\Gamma_I$, namely, 
$$
\begin{aligned}
\sum_{i=1}^5 T_i & = - \int_{\Gamma_I} \bm \sigma_{e} (\bm u_{e})  \bm n_e \cdot \bm w_{e} \,ds 
- \int_{\Gamma_I}  \bm \sigma_{e} (\bm u_{p})  \bm n_p \cdot \bm w_{p} \,ds  \nonumber + \delta \int_{\Gamma_I} \beta p(\bm u_{p},\bm u_{f}) \bm w_{p} \cdot \bm n_p \,ds \\ & + (1-\delta) \int_{\Gamma_I}  \beta p(\bm u_{p},\bm u_{f}) \bm w_{p} \cdot \bm n_p \,ds + \int_{\Gamma_I} p(\bm u_{p},\bm u_{f}) \bm w_{f} \cdot \bm n_p \,ds, 
\end{aligned}
$$
removing the subscript $h$ to ease the notation.
Next, we notice that, in view of condition \eqref{eq::contstress_elporo}, the terms $T_1 + T_2 + T_3$  can be rewritten as 
\begin{align}\label{eq::terms_on_gamma_I}
    T_1 + T_2 + T_3 & = - \int_{\Gamma_I} \bm \sigma_{e} (\bm u_e)  \bm n_p \cdot  (\bm w_p -\bm w_e) \,ds,
\end{align}
while $T_4+T_5$ as 
\begin{align}\label{eq::terms2_on_gamma_I}
    T_4 + T_5 & = - \int_{\Gamma_I} m \nabla \cdot (\beta  \bm u_{p} +   \bm u_{f}) ((1-\delta) \beta \bm w_{p} + \bm w_f) \cdot \bm n_p \,ds.
\end{align}
Then, we add to \eqref{eq::terms_on_gamma_I} and \eqref{eq::terms2_on_gamma_I} the following strongly consistent terms to ensure the symmetry and positivity of the resulting system: 
\begin{align}\label{eq::terms_on_gamma_I_stab}
    - \int_{\Gamma_I} \bm \sigma_{e} (\bm v_e) \bm n_p \cdot  (\bm u_p -\bm u_e) \,ds + \int_{\Gamma_I} \alpha (\bm u_p -\bm u_e) \cdot (\bm v_p -\bm v_e) \,ds,
\end{align}
and 
\begin{multline}\label{eq::terms2_on_gamma_I_stab}
    - \int_{\Gamma_I} m \nabla \cdot (\beta  \bm w_{p} +   \bm w_{f}) ((1-\delta) \beta \bm u_{p} + \bm u_f) \cdot \bm n_p \,ds  \\ + \gamma \int_{\Gamma_I} ((1-\delta) \beta \bm u_{p} + \bm u_f) \cdot \bm n_p  ((1-\delta) \beta \bm v_{p} + \bm v_f) \cdot \bm n_p \,ds,
\end{multline}
being $\alpha$ and $\gamma$ defined as in \eqref{eq::stab_1}--\eqref{eq::stab_2}. 
Finally, summing up equations \eqref{eq::terms_on_gamma_I}--\eqref{eq::terms2_on_gamma_I_stab} we get 
\begin{multline*}
    \mathcal{C}_{h}(\UU,\WW) = - \langle \bm \sigma_{e} (\bm u_e)  \bm n_p, \bm w_p -\bm w_e \rangle_{\mathcal{F}_h^{I}} - \langle \bm \sigma_{e} (\bm v_e) \bm n_p,  \bm u_p -\bm u_e \rangle_{\mathcal{F}_h^{I}} + \langle \alpha \bm u_p -\bm u_e, \bm v_p -\bm v_e \rangle_{\mathcal{F}_h^{I}}  
    \\
    - \langle m \nabla \cdot (\beta  \bm u_{p} +   \bm u_{f}), ((1-\delta) \beta \bm w_{p} + \bm w_f) \cdot \bm n_p \rangle_{\Gamma_I}   
    \\  
    - \langle m \nabla \cdot (\beta  \bm w_{p} +   \bm w_{f}),  ((1-\delta) \beta \bm u_{p} + \bm u_f) \cdot \bm n_p \rangle_{\mathcal{F}_h^{I}}  
    \\ 
    + \gamma \langle ((1-\delta) \beta \bm u_{p} + \bm u_f) \cdot \bm n_p,  ((1-\delta) \beta \bm v_{p} + \bm v_f) \cdot \bm n_p \rangle_{\mathcal{F}_h^{I}},
\end{multline*}
that is definition \eqref{def::bilinear_gammai}. Note that 
$\mathcal{C}_{h}(\UU,\WW) = \mathcal{C}_{h}(\WW,\UU)$.


\end{document}